\documentclass[a4paper,reqno,12pt]{amsart}

\usepackage[a4paper,margin=1in]{geometry}
\usepackage{graphicx}
\usepackage{amsmath,amssymb,amsthm}
\usepackage{hyperref}

\theoremstyle{plain}
\newtheorem{theorem}{Theorem}[section]
\newtheorem{lemma}[theorem]{Lemma}
\newtheorem{proposition}[theorem]{Proposition}

\newtheorem{corollary}[theorem]{Corollary}
\theoremstyle{definition}
\newtheorem{remark}[theorem]{Remark}
\numberwithin{equation}{section}

\title[On non-uniqueness of rotational solitary waves]{On non-uniqueness of solitary waves on two-dimensional rotational flow}

\author{Vladimir Kozlov$^a$$^b$}
\address{$^a$Department of Mathematics and Computer Sciences, St Petersburg State University, St Petersburg, Russia,}
\address{$^b$Department of Mathematics, Link\"oping University, Link\"oping, Sweden}

\begin{document}
	
\begin{abstract}

We consider solitary water waves on a rotational, unidirectional flow in a two-dimensional channel of finite depth.
Ovsyannikov has conjectured in 1983 that the solitary wave is uniquely determined by the Bernoulli constant,  mass flux and by the flow force.  This conjecture was disproved by Plotnikov in 1992 for the ir-rotational flow. In this paper we show that this conjecture is wrong also for rotational flows. Moreover we prove that in any neighborhood of the first bifurcation point on the branch of solitary waves, approaching the extreme wave, there are infinitely many pairs of solitary waves corresponding to the same Bernoulli constant. We give a description of the structure of this set of pairs.
The proof is based on a bifurcation analysis of the global branch of solitary waves which is of independent interest.

\end{abstract}

\maketitle

\section{Introduction}

We consider steady water waves in a two-dimensional channel bounded below by a flat,
rigid bottom and above by a free surface that does not touch the bottom. We use the classical formulation of the problem based on the Euler equations. The surface tension is neglected and the water motion can be rotational. The formulation of this problem contains three constants the mass flux $\widehat{Q}$, the Bernoulli constant $R$ and the flow force $S$. The main subject of our study is the solitary waves. They are characterized by the asymptotics at $\pm\infty$, which is a laminar flow. In  Kozlov, Lokharu, Wheeler \cite{KLW} it was proved that such laminar flows are supercritical, i.e. the corresponding Froude number is greater than $1$.
We assume that the mass flux $\widehat{Q}$ is fixed. Then all supercritical laminar flows can be uniquely parametrized by the Bernoulli constant $R$. We denote by $S_-(R)$ the corresponding flow force. Then the flow force corresponding to the solitary wave equals $S_-(R)$. It was proved in Lokharu \cite{Lok1} that if a certain solution to the general water wave problem with the Bernoulli constant $R$ has the flow force $S$ and $S=S_-(R)$ then this solution is a solitary wave (or a supercritical laminar flow). Ovsyannikov \cite{Ov}  has conjecture in 1983 that  solitary waves  in the ir-rotational case are uniquely determined by the Bernoulli constant. This conjecture was disproved by Plotnikov \cite{P3} in 1992. In this paper we show that this conjecture is wrong for rotational flows also. Moreover we show that  there are infinitely many pairs of solitary waves corresponding to the same Bernoulli constant in any neighborhood of the first bifurcation point on the branch of solitary waves, approaching the extreme wave.

Here we can not use Plotnikov's approach based on application of methods from complex analysis. We develop a new technique which is based on a bifurcation analysis of the global branch of the solitary waves.

\subsection{Statement of the main result}

Steady water waves
in an appropriate coordinate system $(X,Y)$ moving along with the wave with the same velocity $c$,
can be described by the following stationary Euler equations  independent of time:
\begin{equation}\label{Se17a}
u_X+v_Y=0\;\;\mbox{in $D$}
\end{equation}
and
\begin{eqnarray}\label{Se17aa}
&&(u-c)u_X+vu_Y+P_X=0\;\;\mbox{in $D$},\nonumber\\
&&(u-c)v_X+vv_Y+P_Y=-g\;\;\mbox{in $D$}.
\end{eqnarray}
where
$$
D=D_\xi=\{(X,Y)\,:\,0<Y<\xi(X),\;X\in{\mathbb R}\}
$$
describes the flow domain in the channel with  the flat bottom $B$ given by $Y=0$  and with the free surface $S=S_\xi$ given by $Y=\xi(X)$. Here $(u,v)$ is the velocity vector of the fluid, $P$ is the pressure and $g$ is the gravitational constant.

There are two boundary conditions on the free surface. One expresses the fact that each water
particle on the surface remains on it. The other states that on the surface the pressure
$P$ equals the atmospheric pressure $P_{atm}$. So
\begin{equation}\label{Sep18a}
v=(u-c)\xi_X\;\;\mbox{and}\;\;P=P_{atm}\;\;\mbox{on $S$}.
\end{equation}
The bottom is impermeable:
\begin{equation}\label{Sep18aa}
v=0\;\;\mbox{on $B$}.
\end{equation}

We  introduce the stream function $\Psi=\Psi(X,Y)$ by
$$
\Psi_X(X,Y)=v(X,Y),\;\;\Psi_Y(X,Y)=c-u(X,Y).
$$
Then  equation  (\ref{Se17a}) is satisfied.
The vorticity function $v_X-u_Y$ is constant on the curves $\Psi=\mbox{const}$. Therefore there exist a function  $\omega(\Psi)$ such that $v_X-u_Y=\omega(\Psi)$. We will consider unidirectional flow, i.e. $\Psi_Y$ has the same sign in the flow, which implies that $\omega$ is single-valued.
Now the  equations in (\ref{Se17aa}) can be written as
$$
\frac{1}{2}(\Psi_X^2+\Psi_Y^2)+ P+gY+ \Omega(\Psi)=C,
$$
where
$$
\Omega(\Psi)=\int_0^\Psi \omega(\tau)d\tau.
$$
So if we found $\Psi$ we can use this formula to find $P$.
Therefore the boundary condition $P=P_{atm}$ becomes
$$
\frac{1}{2}(\Psi_X^2+\Psi_Y^2)+g\xi(X)=R\;\;\mbox{on $S$},
$$
where $R$ is a constant, which is called the Bernoulli constant (the total head).

 The first boundary conditions in (\ref{Sep18a}) and (\ref{Sep18aa}) takes the form
\begin{equation*}
\partial_X\Psi(X,\xi(X))=0\;\;\mbox{on $S$ }\;\;\mbox{and}\;\;\partial_X\Psi(X,0)=0
\end{equation*}
for all $X$.
 Since the function $\Psi$ is defined up to a constant, we assume that $\Psi=0$ on $B$. Then  $\Psi$ is a constant on $S$, say $\widehat{Q}$, which is called the mass flux and
$$
\widehat{Q}=\int_0^{\xi(X)}(c-u)dY.
$$
 After a scaling we can assume that $\widehat{Q}=1$ and $g=1$, and we arrive at the following boundary value problem
\begin{eqnarray}\label{K2a}
&&\Delta \Psi+\omega(\Psi)=0\;\;\mbox{in $D_\xi$},\nonumber\\
&&\frac{1}{2}|\nabla\Psi|^2+\xi=R\;\;\mbox{on $S_\xi$},\nonumber\\
&&\Psi=1\;\;\mbox{on $S_\xi$},\nonumber\\
&&\Psi=0\;\;\mbox{for $Y=0$}.
\end{eqnarray}
 There is one more constant
the flow force
\begin{equation}\label{Nov20aa}
{\mathcal S}=\int_0^{\xi}(P-P_{atm}+U^2)dY=\int_0^\xi \Big(\frac{1}{2}\Big(\Psi_Y^2-\Psi_X^2\Big)+\Omega(\widehat{Q})-\Omega(\Psi)-Y+R\Big)dY.
\end{equation}
The importance of these  constants in classifying the possible steady  flows along a uniform horizontal channel, pointed out by
Benjamin and Lighthill, \cite{BL54}. Another important tool in classification of water waves are laminar flows and conjugate flows. They represents solutions independent of the horizontal variable $X$. If the flux and the Bernoulli constant are fixed then there exist two laminar flows one of them has  depth $d_-(R)$ and is called supercritical, the second one has a depth $d_+(R)$ and is called subcritical. Here $d_-(R)<d_+(R)$. The corresponding flow forces are denoted by $S_-(R)$ and $S_+(R)$ respectively.

We assume that $\omega\in C^{1,\alpha}$, $\alpha\in (0,1)$, $\Psi$ and $\xi$ are even in $X$ and
\begin{equation}\label{J27ba}
\Psi_Y>0\;\;\mbox{on $\overline{D_\xi}$},
\end{equation}
which means that the flow is unidirectional.

The main subject of our study is the solitary waves. The solitary wave solution satisfies
\begin{equation}\label{Ju31a}
\;\;\xi(X)\to d,\;\;\Psi(X,Y)\to U(Y)\;\;\mbox{as $X\to\pm\infty$}
\end{equation}
uniformly in $Y$. Here $(U,d)$ is a uniform stream solution (laminar flow) with the Bernoulli constant  $R=d+U'(d)^2/2$. Moreover the Froude number
\begin{equation}\label{Au22b}
F=\Big(\int_0^d\frac{dY}{U'(Y)^2}\Big)^{-2}
\end{equation}
for this uniform stream solution  must be $>1$.  Every non-trivial
solitary wave satisfies $\xi(X)>d$ and it is symmetric, after a certain
translation, and monotonic decreasing for positive $X$. We refer to \cite{KLW} for the proof that all solitary waves are supercritical ($F$ is greater than $1$), where can be found  references for other properties of the solitary waves listed above.

If the Froude number $F$ is equal to $1$ or equivalently $R=R_c$, where $R_c$ is minimal Bernoulli constant for the uniform stream solutions, then the only solution of (\ref{K2a}) is the uniform stream solution $(U_*(Y),d_*)$ corresponding to the Bernoulli constant $R_c$.
 For solitary waves with asymptotics (\ref{Ju31a}) the flow force can be written more explicitely as
\begin{equation}\label{Sep29aa}
{\mathcal S}={\mathcal S}(R)=\int_0^{d}\Big(\frac{1}{2}U_Y^2-\Omega(U)+\Omega(1)-Y+R\Big)dY.
\end{equation}
One can show that this quantity strongly increases when $R$ increases if we choose the uniform stream solution with $F>1$, see Sect.\ref{SSep28}.

It was proved in \cite{KLW} that this laminar flow has the flow force $S_-(R)$, where $R$ is the same as for the solitary wave. Moreover any solution to (\ref{K2a}) with flow force $S_-(R)$ is a solitary wave (or a supercritical laminar flow), see  \cite{Lok1}. Another important result on solitary waves is due to Wheeler \cite{W}. He has proved existence of the global branch of solitary waves starting from a supercritical laminar flow corresponding to the critical value of the Bernoulli constant $R_c$ (the minimal value of the Bernoulli constant for the laminar flows). In   Chen, Walsh and Wheeler \cite{W2} it was explained that such branch can be chosen to be analytic up to local re-parameterization and it is unique if we will keep this analyticity property. So this branch is uniquely connected to the vorticity. We assume that the vorticity $\omega$ satisfies the following property
\begin{eqnarray}\label{Bov20a}
&&\mbox{The analytic branch corresponding to the vorticity $\omega$ approaches} \nonumber\\
&&\mbox{an extreme wave with the angle $120^\circ$ at its crest}
\end{eqnarray}
In Remark \ref{RNov22a} to Proposition \ref{Pr26} it is explained when this can happen.
One of the main results proved here is the following

\begin{theorem}\label{T12-2} Let $\omega$ satisfy (\ref{Bov20a}). Then in any neighborhood of the first bifurcation point on the branch of solitary wave solutions to the problem (\ref{K2a}), approaching the extreme wave, there are infinitely many pairs of solitary waves corresponding to the same Bernoulli constant.
\end{theorem}

A proof of this assertion follows from Theorem \ref{Tsep27}, which is proved  in Sect. \ref{SSep28}.

\subsection{Main steps of the proof. Global branch of solitary waves and its first bifurcation.}

Let us introduce several function spaces. For $\alpha\in (0,1)$ and $k=0,1,\ldots$, the space $C^{k,\alpha}(D)$ consists of bounded  functions in $D$ such that the semi-norms \\ $||\cdot||_{C^{k,\alpha}(D_{(a,a+1)})}$ are uniformly bounded with respect to $a\in{\mathbb R}$. Here
$$
D_{a,a+1}=\{(X,Y)\in \overline{D},\;:\,a\leq X\leq a+1\}.
$$
The space    $C^{k,\alpha}_{0}(D)$  consists of function in $C^{k,\alpha}(D)$  vanishing for $Y=0$. The space $C^{k,\alpha}_{0, e}(D)$ consists of even functions from $C^{k,\alpha}_0(D)$.
Similarly we define the space    $C^{k,\alpha}_{ e}(\Bbb R)$) consisting of even functions in $C^{k,\alpha}(\Bbb R)$.

We will consider a branch of solitary  waves depending on a parameter $t> 0$, i.e.
\begin{equation}\label{J3b}
\Psi=\Psi(X,Y;t),\;\;\xi=\xi(X,t),\;\;R=R(t),\;\;t\in (0,\infty),
\end{equation}
where $R$ is the Bernoulli constant.

We assume that the above branch of solitary waves satisfies the conditions of the following theorem, see \cite{W} and \cite{Koz1b}.

\begin{theorem}\label{T1}  Fix a H{\"o}lder exponent $\alpha\in(0, 1/2]$. There exists a continuous curve {\rm (\ref{J3b})}
of solitary wave solutions to {\rm (\ref{K2a})} with the regularity
\begin{equation}\label{Au30a}
(\Psi(t),\xi(t), R(t))) \in C^{2,\alpha}_0( D(t))\times C^{2,\alpha}({\Bbb R})\times (R_c,\infty),
\end{equation}
where $D(t):=D_{\xi(t)}$. The solution curve satisfies  the following property (critical laminar flow, i.e. $R=R_c$)
\begin{equation}\label{Ju18aa}
\lim_{t\to 0}(\Psi(t),\xi(t),R(t))=(U_*,d_*,R_c).
\end{equation}
The convergence is understood in the spaces
$$
C^{2,\alpha}(D(t)) \times C^{2,\alpha}({\Bbb R})\times \Bbb R.
$$

Furthermore, there exists a sequence $\{t_j\}_{j=1}^\infty$ such that $t_j\to\infty$ as $j\to\infty$ and $(X_j,Y_j)\in D(t_j)$ such that the following  is valid
\begin{equation}\label{Au25a}
 (\Psi(t_j))_Y(X_j,Y_j)\to 0\;\; \mbox{as $j\to\infty$}\;\;\;\mbox{(Stagnation)}.
 \end{equation}

\end{theorem}

\begin{remark}

{\rm (i)} The dependence on $t\in (0,\infty)$ in (\ref{Au30a}) is analytic in the variables $(q,p)$ using in partial hodograph transformation, see \cite{W2}, Sect.6.

{\rm (ii)} In paper \cite{W} the Froude number $F(t)$ is taken instead of $R(t)$. Since the dependence $R$ on $F$ is monotone and analytic (see \cite{Koz1b}), we use here the function $R(t)$.
\end{remark}

In the  next proposition  we present a clarification of the stagnation condition (\ref{Au25a}) by replacing it by conditions containing stagnation points on the free surface or on the bottom and by overhanging points on the free surface. Introduce the notations
\[
\theta_0=\sqrt{2\max_{\tau\in[0,1]}\Omega(\tau)},\;\;R_0={\mathcal R}(\theta_0),
\]
where
\[
{\mathcal R}(\theta)=\frac{1}{2}\theta^2+\int_0^1\frac{d\tau}{\sqrt{\theta^2-2\Omega(\tau)}}-\Omega(1).
\]

\begin{proposition}\label{Pr26} The stagnation condition {\rm (\ref{Au25a})} is equivalent to one of  the following options

{\rm (i)} If $R_0=\infty$ or $R_0<\infty$ and $\Omega(1)>0$ then  there exists a sequence $\{t_j\}_{j\geq 1}$, $t_j\to\infty$ as $j\to\infty$  such that
\begin{equation}\label{Au29ba}
\lim_{j\to\infty}|\xi'(t_j)(0)|=\infty\;\;\mbox{overhanging wave}
\end{equation}
or
\begin{equation}\label{Au29bb}
\lim_{j\to\infty}R(t_j)=R\;\;\mbox{and}\;\;\lim_{j\to\infty}\xi(t_j)(0)=R.\;\;\mbox{surface stagnation}
\end{equation}

{\rm (ii)} If $R_0<\infty$ and $\theta_0=0$ then besides the options in {\rm (i)} there is one more option:
\begin{equation}\label{Au29bc}
\lim_{j\to\infty}R(t_j)=R\;\;\mbox{and}\;\;\lim_{j\to\infty}\sup_{X\in\Bbb R}\Psi_Y(t_j)(X_j,Y_j)=0,\;\;\mbox{bottom stagnation}
\end{equation}
where $Y_j\to 0$ as $j\to \infty$.

\end{proposition}
 A proof of this proposition is given in Remark \ref{R21a}.

\begin{remark}\label{RNov22a} Consider the case (i) in the above proposition. In the paper Strauss, Wheeler \cite{StrWh} a certain condition on the vorticity $\omega$
 was suggested,  which  guarantees that the property (\ref{Au29ba}) fails.

 The case (\ref{Au29bb}) is closely related to the assumption (\ref{Bov20a}). As it is shown in \cite{KL1} (see also \cite{Koz1}) there exists a weak limit of solitary waves approaching an extreme wave with the angle $120^\circ$ or $180^\circ$ at the crest. 

\end{remark}

The Frechet derivative for the problem in the periodic case, written after partial hodograph change of variables, is evaluated in  \cite{Koz1}, where an equivalent form for the Frechet derivative in $(X,Y)$ variables is presented also. The same evaluation can be used to obtain the Frechet derivative in the case, when the solution is a solitary wave.  The corresponding eigenvalue problem for the  Frechet derivative in $(X,Y)$ variables has the form
\begin{eqnarray}\label{J17ax}
&&\Delta w+\omega'(\Psi)w+\mu w=0 \;\;\mbox{in $D_\xi$},\nonumber\\
&&\partial_n w-\rho w=0\;\;\mbox{on $S_\xi$},\nonumber\\
&&w=0\;\;\mbox{for $Y=0$},
\end{eqnarray}
where $n$ is the unite outward normal to $S_\xi$, $w$ is a periodic, even function and
\begin{equation}\label{Sept17aa}
\rho=
\rho(X)=\frac{(1+\Psi_X\Psi_{XY}+\Psi_Y\Psi_{YY})}{\Psi_Y(\Psi_X^2+\Psi_Y^2)^{1/2}}\Big|_{Y=\xi(X)}.
\end{equation}

Denote by ${\mathcal A}(t)={\mathcal A}_{(\Psi,\xi)}$ the unbounded operator in $L^2(D_\xi)$ corresponding to the problem (\ref{J17ax}), i.e.
${\mathcal A}_{(\Psi,\xi)}=-\Delta -\omega'(\Psi)$ with the domain
\begin{eqnarray*}
&&{\mathcal D}(t)={\mathcal D}({\mathcal A}_{(\Psi,\xi)})=\{w\in H^2(D_\xi)\,:\,w(X,0)=0,\;\;\partial_\nu w-\rho w=0\, \mbox{for}\; Y=\xi(X)\\
&&\;\;\mbox{and $w$ is even}\}.
\end{eqnarray*}
For every $t>0$ we denote by
\begin{equation}\label{Au27b}
(U(t),d(t),R(t))
\end{equation}
 the limit uniform stream solution and the corresponding Bernoulli constant in the asymptotics (\ref{Ju31a}) of solution (\ref{J3b}).

To describe spectrum of the problem (\ref{J17ax}) we introduce  the one-dimensional eigenvalue problem
\begin{eqnarray}\label{Okt25az}
&& -v^{''}-\omega'(U(t))v=\nu v\;\;\mbox{on $(0,d(t))$},\nonumber\\
&&v(0)=0\;\;\mbox{and}\;\;v'(d(t))-\rho_0(t)v(d(t))=0,
\end{eqnarray}
where
\begin{equation}\label{Okt25a}
\rho_0(t)=\frac{1+U_YU_{YY}}{U_Y^2}\Big|_{Y=d(t)}.
\end{equation}
In Proposition \ref{PrAp8a} it is proved that the eigenvalues of this problem are positive for all $t>0$, since the Froude number corresponding to the uniform stream solution (\ref{Au27b}) is larger than $1$. We denote by $\nu_0=\nu_0(t)$ the lowest eigenvalue.  This eigenvalue is simple and corresponding eigenfunction does not change sign in $(0,d(t)]$.

The following properties of the operator ${\mathcal A}$ can be easily verified and there proofs can be found in \cite{Koz1b}.
\begin{proposition}
{\rm (i)} For every $t>0$ the operator   ${\mathcal A}$ has a continuous spectrum $[\nu_0(t),\infty)$. The spectrum on the half line $(-\infty,a]$ consists of finite number of isolated eigenvalues of finite multiplicity for any $a<\nu_0(t)$.

{\rm (ii)} The lowest eigenvalue of the operator ${\mathcal A}(t)$ for $t>0$  is always negative and simple with the eigenfunction which does not change sign in $D(t)$.
We denote it by $\mu_0(t)$.

{\rm (iii)} For small positive $t$ the spectrum of the operator ${\mathcal A}(t)$ on the half line $\mu\leq 0$ consists of the  eigenvalue $\mu_0(t)$ only.
\end{proposition}

The next result is important for our study of the first bifurcation and its proof  can be found in \cite{Koz1b}

\begin{theorem}\label{T1a}
If for a certain $t_1>0$ and $t\in (0,t_1)$ the operator ${\mathcal A}(t)$  has the only negative eigenvalue $\mu_0$ and $\mu=0$ is the eigenvalue of ${\mathcal A}(t_1)$  then this eigenvalue is simple.

\end{theorem}

Denote by $\mu_1(t)$ the second eigenvalue of the operator ${\mathcal A}(t)$ on the half-line $(-\infty,\nu_0)$ if it exists, otherwise we put $\mu_1(t)=\nu_0(t)$.
Assume that
\begin{eqnarray}\label{Okt30a}
&&\mbox{there exists a point $t_*>0$ such that: (i) $\mu_1(t)\geq 0$ for $t\in (0,t_*)$},\nonumber\\
&&\mbox{ $\mu_1(t)<0$ for small positive $t-t_*$}.
\end{eqnarray}
According to Theorem \ref{T1a}(iii) the eigenvalue $\mu_1(t_*)$, which is equal to $0$, is simple. This implies that the crossing number of the Frechet derivative is $1$ at $t_*$.
Since the eigenvalue $\mu_1(t)$ is analytic and due to (\ref{Okt30a}) $\mu_1(t)$ changes sign at $t_*$. Therefore it can be represented as
\begin{equation}\label{F18a}
\mu_1(t)=\sum_{j=m}^\infty \hat{\mu}_j(t-t_*)^j,
\end{equation}
where $m$ is odd and $\hat{\mu}_m<0$.

In the next proposition a condition, when (\ref{Okt30a}) is fulfilled, is given.
 \begin{proposition}\label{PrNov11} \rm {(\cite{Koz1b})}
  Assume that there exists a sequence $\{t_j\}_{j=1}^\infty$ such that the curve {\rm (\ref{J3b})} approaches an extreme solitary wave with the angle
$120^\circ$  at the crest along this sequence. Then there exists a point $t_*$ which satisfies the condition {\rm (\ref{Okt30a})}.
 \end{proposition}

The aim of this paper is a description of the first bifurcation on the curve (\ref{J3b})  in more details. The corresponding bifurcation problem can be written as
\begin{eqnarray}\label{K2az}
&&\Delta \widehat{\Psi}+\omega(\widehat{\Psi})=0\;\;\mbox{in $D_{\widehat{\xi}}$},\nonumber\\
&&\frac{1}{2}|\nabla\widehat{\Psi}|^2+\widehat{\xi}=R(t)\;\;\mbox{on $S_{\widehat{\xi}}$},\nonumber\\
&&\widehat{\Psi}=1\;\;\mbox{on $S_{\widehat{\xi}}$},\nonumber\\
&&\widehat{\Psi}=0\;\;\mbox{for $Y=0$},
\end{eqnarray}
where $\widehat{\Psi}=\Psi(t)+\tilde{w}$ and $\widehat{\xi}=\xi(t)+\tilde{\xi}$.
Now we are looking for branches of solutions to this problem passing through the solution $(\Psi(t_*),\xi(t_*),R(t_*))$ and different from (\ref{J3b}). Observe that in this problem we have a new parameter $t$ instead of $R$ as before.

The following theorem contains the description of the first bifurcation of the problem (\ref{K2az}).

\begin{theorem}\label{ThF18} Let condition {\rm (\ref{Okt30a})} and hence condition  {\rm (\ref{F18a})} hold. There exist at least one and at most $m$ continuous curves
$$
(\widehat{\Psi}_s,\widehat{\xi}_s,t(s))\in C^{2,\alpha}(D_{\widehat{\xi}})\times C^{2,\alpha}(\Bbb R)\times (0,\infty),\;\;\;s\in [0,\infty),
$$
bifurcating from $(\Psi(t_*),\xi(t_*),t_*)$. These curves after partial hodograph transformation (see Sect. \ref{SJ29a})  admit an analytic re-parameterisation near each point $s\geq 0$. Furthermore,  $t(0)=t_*$, $(\widehat{\Psi}_0, \widehat{\xi}_0)=(\Psi(t_*),\xi(t_*))$. Moreover each curve can be one of the following:

{\rm (i)} $t(s)=t_*$ for all $s\geq 0$;

{\rm (ii)} $\mu_1((t(s))=0$ for all $s\geq 0$;

{\rm (iii)} $t'(s)\neq 0$   and the Frechet derivative is invertible  for small positive $t-t_*$;

Furthermore the bifurcation curve satisfies on of the following properties: 

 {\rm (a)} $\sup_{s>0,(X,Y)\in D_{\widehat{\xi}_s}}|\widehat{\xi}_s'(X)|=\infty$ (overhanging);

{\rm (b)} $\inf_{s>0,(X,Y)\in D_{\widehat{\xi}_s}}\partial_Y\widehat{\Psi}_s(X,Y)=0$ (stagnation);

{\rm (c)} $t(s)\to\infty$\;\;\mbox{as $s\to\infty$};

{\rm (d)} the branch is a periodic curve.

\end{theorem}

This paper is mostly devoted to the proof of Theorem \ref{ThF18}, where the structure of global bifurcating branches
starting from the bifurcation point $(\Psi(t_*),\xi(t_*),R(t_*))$ is described. Our approach is based on an abstract study of bifurcations at a simple eigenvalue, which models nonlinear boundary value problems. We upgrade known results on the local and global structure of bifurcating solutions and obtain a description of global structure of the set of bifurcating solution at a simple eigenvalue.
We note that we can guarantee that the point $t_*$ gives the secondary bifurcation only when this is not a turning point, i.e. the function $R(t)$ is monotone in a neighborhood of the bifurcation point. Indeed, bifurcation in the frame of Theorem \ref{ThF18} means that if $(\Psi(t(s_1),\xi(t(s_1),t(s_1))=(\Psi(t_2),\xi(t_2),t_2)$ for small $s_1$ and $t_2-t_*$ then $s_1=0$ and hence $t(s_1)=t_*$. Secondary bifurcation, which relates to  the curve (\ref{J3b}), means that if $(\Psi(t(s_1),\xi(t(s_1),R(t(s_1)))=(\Psi(t_2,\xi(t_2),R(t_2))$ for small $s_1$ and $t_2-t_*$ then $s_1=0$ and hence $t(s_1)=t_*$. Due to local monotonicity near $t_*$  and analyticity of $R$ this function is strongly monotonic near $t_*$. Therefore $t(s_1)=t_2$ and the bifurcation property in the above theorem gives $s_1=0$, which delivers secondary bifurcation branches.

\begin{theorem}\label{Tsep27} Let the condition (\ref{Bov20a}) be fulfilled. Then there exist at least two different solitary waves with the same flow force constant.
Moreover the following properties are valid

(1) If $t_*$ is a turning point, i.e. the function $R(t)$ is not monotone in any neighborhood of the point $t_*$, then there exists $\varepsilon>0$ such that for every
\begin{eqnarray*}
&&R\in (R_*-\varepsilon,R_*),\;\;\mbox{if $R(t)$ is increasing-decreasing or}\\
&&R\in (R_*,R_*+\varepsilon),\;\;\mbox{if $R(t)$ is decreasing-increasing},\;\;R_*=R(t_*),
\end{eqnarray*}
 the equation
$R(t)=R$ has two roots $t_1(R)$, $t_2(R)$ continuously depending on $R$ and $t_k(R_*)=t_*$ and the vector-functions $(\Psi(t_k(R),\xi(t_k(R))$, $k=1,2$, solves the problem (\ref{K2a})
and they are different for each $R$ in the interval.

(2) The function $R(t)$ is strongly monotone in a neighborhood of $t_*$. Then

(i) If $t(s)=t_*$ then $(\widehat{\Psi}_s,\widehat{\xi}_s,R)$ delivers a curve of solutions to (\ref{K2az}), which are different from the solutions (\ref{J3b}) in a neighborhood of $s=0$ exept the point $s=0$ where they coinside.

(ii) If the function $t(s)$ is not a constant then it is strongly monotone for $s\in (0,\epsilon)$ for small $\epsilon$. Let it be strongly increasing. Then the equation $t(s)=t$ has a unique root for small positive $t-t_*$, we denote it by $s(t)$. Then for the Bernoulli constant $R(t)$ we have two solutions
$$
(\Psi(t),\xi(t))\;\;\mbox{and}\;\; (\widehat{\Psi}_{s(t)},\widehat{\xi}_{s(t)}).
$$
If $t(s)$ be strongly decreasing then the same is true for $s\in(-\epsilon,0)$.
\end{theorem}

One can verify that Theorem \ref{T12-2} follows from the above result.

The first study of bifurcations on the branches of solitary water waves was done by Plotnikov in \cite{P3} in the irrotational case. In the case of vortical flows one cannot apply the same technique as in the irrotational case since it is based on the complex analysis approach. We note that even in the ir-rotational case we obtain the existence of infinitely many pairs of solitary solutions with the same Bernoulli constant in any neighborhood of the first bifurcation point, which is  also a new result.

 The important contributions allowing to approach the  bifurcation problems are papers  \cite{W}, \cite{W2} and \cite{Koz1b}, where the global branches of solitary waves were constructed, and the papers \cite{KL2} together with \cite{Koz1} and \cite{Koz1a}, where the bifurcation analysis for branches of Stokes waves were presented.

\subsection{Uniform stream solutions and the dispersion equation}\label{SJ6a}

The uniform stream solution $\Psi=U(Y)$ with the constant depth $\xi =d$  satisfies the problem
\begin{eqnarray}\label{X1}
&&U^{''}+\omega(U)=0\;\;\mbox{on $(0,d)$},\nonumber\\
&&U(0)=0,\;\;U(d)=1,\nonumber\\
&&\frac{1}{2}U'(d)^2+d=R.
\end{eqnarray}
To find solutions to this problem we introduce a parameter $\theta=U'(0)$. We assume that
 $\theta>\theta_0:=\sqrt{2\max_{\tau\in [0,1]}\Omega(\tau)}$, where
\begin{equation}\label{M19a}
\Omega(\tau)=\int_0^\tau \omega(p)dp.
\end{equation}
Then the problem (\ref{X1}) has a solution $(U,d)$ with a strongly monotone function $U$ for
\begin{equation}\label{M6ax}
R={\mathcal R}(\theta):=\frac{1}{2}\theta^2+d(\theta)-\Omega(1).
\end{equation}
This solution is given by
\begin{equation}\label{F22a}
Y=\int_0^U\frac{d\tau}{\sqrt{\theta^2-2\Omega(\tau)}},\;\;d=d(\theta)=\int_0^1\frac{d\tau}{\sqrt{\theta^2-2\Omega(\tau)}}.
\end{equation}
The function $d$ is strongly decreasing. 

If we consider (\ref{M6ax}) as the equation with respect to $\theta$ then it is solvable if $R\geq R_c$, where
\begin{equation}\label{F27a}
R_c=\min_{\theta\geq \theta_0}{\mathcal R}(\theta),
\end{equation}
and it has two solutions if
$R\in (R_c,R_0)$, where
\begin{equation}\label{D19ba}
R_0={\mathcal R}(\theta_0).
\end{equation}
We denote by $\theta_c$ the point where the minimum in (\ref{F27a}) is attained. Clearly, the supercritical laminar flows (flows with $\theta>\theta_c$) can be uniquely parameterized by the Bernoulli constant $R>R_c$.

The following proposition describes a connection between the Bernoulli constant $R$, the Froude number $F$ and the supercritical flow force $S_-$.
\begin{proposition}\label{Pr26b}
(i) The Bernoulli constant $R\in [R_c,\infty)$ can be considered as a continuous function of the Froude number $F\in [1,\infty)$. This function is strongly increasing, analytical on $(1,\infty)$, $R(1)=R_c$ and $R(F)\to\infty$ as $F\to\infty$.

(ii) $S_-(R)$ is a strongly monotonic function for $R\geq R_c$.


(iii) There exists a constant $\widehat{R}$ such that $R\leq \widehat{R}$ for all $R$ corresponding to  solitary waves.
\end{proposition}
\begin{proof} The assertion (i) is proved in \cite{Koz1b}.

(ii) The uniform stream solution $(U(y;\theta),d(\theta))$, describing the asymptotics (\ref{Ju31a}) of the solitary waves at $\pm\infty$, is given by (\ref{F22a}) with $\theta>\theta_c$. Let us show that the supercritical flow force is increasing function of $R$.

The flow force for $(U(y;\theta),d(\theta))$ is defined by
$$
{\mathcal S}(\theta)=\int_0^d\Big(\frac{1}{2}U'^2-\Omega(U)+\Omega(1)-Y+{\mathcal R}(\theta)\Big)dy.
$$
Then
\begin{eqnarray*}
&&\partial_\theta {\mathcal S}(\theta)=\Big(\frac{1}{2}U_Y^2(d;\theta)-d+R\Big)d'+\int_0^d\Big(U_YU_{Y\theta}-\omega(U)U_\theta+{\mathcal R}'(\theta)\Big)dY\\
&&=\Big(\frac{1}{2}U_Y^2(d;\theta)-d+R\Big)d'+U_Y(d;\theta)U_\theta(d;\theta)+{\mathcal R}'(\theta)d.
\end{eqnarray*}
Using that
$$
\partial_\theta U(d;\theta)=U_Y(d;\theta)d'+U_\theta(d;\theta)=0,
$$
we obtain
$$
\partial_\theta {\mathcal S}(\theta)=\Big(-\frac{1}{2}U_Y^2(d;\theta)-d+R\Big)d'+{\mathcal R}'(\theta)d={\mathcal R}'(\theta)d.
$$
Therefore ${\mathcal S}$ is increasing for $\theta>\theta_c$. Since the function ${\mathcal R}(\theta)$ is also increasing for $\theta>\theta_c$ and tends to $\infty$ when $\theta\to\infty$, the function ${\mathcal S}$ can be considered as an increasing function of $R$.

(iii) This is proved in \cite{Koz1b}, Proposition 2.4.
\end{proof}

In the next proposition we prove the positivity of the lowest eigenvalue of the problem (\ref{Okt25az}).

\begin{proposition}\label{PrAp8a} Let $(U,d)$ be a supercritical (i.e. $F>1$) laminar flow corresponding to the Bernoulli constant $R>R_c$. Consider the eigenvalue problem
\begin{eqnarray}\label{Ap8a}
&& -v^{''}-\omega'(U)v=\nu v\;\;\mbox{on $(0,d)$},\nonumber\\
&&v(0)=0\;\;\mbox{and}\;\;v'(d)-\rho_0v(d)=0,
\end{eqnarray}
where $\rho_0$ is given by (\ref{Okt25a}).
Then the smallest eigenvalue $\nu_0$ of this eigenvalue problem is positive.

\end{proposition}
\begin{proof}
Consider first the spectral problem
\begin{eqnarray}\label{Ap7a}
&& -w^{''}-\omega'(U)w=0\;\;\mbox{on $(0,d)$},\nonumber\\
&&w(0)=0\;\;\mbox{and}\;\;w'(d)-\rho_0w(d)=\sigma w.
\end{eqnarray}
 Its first eigenfunction is $\gamma(y)$, which solves the first equation in (\ref{Ap7a}), and $\gamma(0)=0$, $\gamma(d)=1$. Then the corresponding eigenvalue is $\sigma=\gamma'(d)-\rho_0 \gamma(d)$. According to \cite{Koz1b}, Sect. 2.1
\begin{equation}\label{Ap8aa}
\sigma=-\frac{3}{2U'(d)}\frac{{\mathcal R}'(\theta)}{d'(\theta)}=\frac{3(F^2(\theta)-1)}{2U'(d)},
´\end{equation}
where ${\mathcal R}(\theta)=R$ and $\theta>\theta_c$.

Multiplying (\ref{Ap7a}) by the solution $v$ to the problem (\ref{Okt25az}), we get
\begin{eqnarray*}
&&0=\int_0^d(-w^{''}-\omega'(U)w)vdy=\int_0^d(-v^{''}-\omega'(U)v)w dy-(w'v)(d)+(wv')(d)\\
&&=\nu_0\int_0^dvwdy-\mu(wv)(d).
\end{eqnarray*}
Therefore
$$
\nu_0\int_0^dv\gamma(y;0)  dy=\mu v(d).
$$
Since both functions $\gamma$ and $v$ do not vanish in $(0,d]$ this implies $\nu_0>0$.
\end{proof}

\subsection{Partial hodograph transform}\label{SJ29a}

In what follows we will study  branches of solitary waves $(\Psi(X,Y;t),\xi(X;t), R(t))$  started from the uniform stream solution at $t=0$. The existence of such branches is established in \cite{W} (see Theorem \ref{T1} here).

We assume that
$$
\Psi_Y>0\;\;\mbox{in $\overline{D_\xi}$}
$$
and use the variables
$$
q=X,\;\;p=\Psi.
$$
Then
$$
q_X=1,\;\;q_Y=0,\;\;p_X=\Psi_X,\;\;p_Y=\Psi_Y,
$$
and
\begin{equation}\label{F28b}
\Psi_X=-\frac{h_q}{h_p},\;\;\Psi_Y=\frac{1}{h_p}.
\end{equation}

System (\ref{K2a}) in the new variables takes the form
\begin{eqnarray}\label{J4a}
&&\Big(\frac{1+h_q^2}{2h_p^2}+\Omega(p)\Big)_p-\Big(\frac{h_q}{h_p}\Big)_q=0\;\;\mbox{in $Q$},\nonumber\\
&&\frac{1+h_q^2}{2h_p^2}+h=R\;\;\mbox{for $p=1$},\nonumber\\
&&h=0\;\;\mbox{for $p=0$}.
\end{eqnarray}
Here
$$
Q=\{(q,p)\,:\,q\in\Bbb R\,,\;\;p\in (0,1)\}.
$$
The uniform stream solution corresponding to the solution $U$ of (\ref{X1}) is
\begin{equation}\label{M4c}
H(p)=\int_0^p\frac{d\tau}{\sqrt{\theta^2-2\Omega(\tau)}},\;\;\theta=U'(0)=H_p^{-1}(0).
\end{equation}
One can check that
\begin{equation}\label{J18aaa}
H_{pp}-H_p^3\omega(p)=0
\end{equation}
or equivalently
\begin{equation}\label{J18aa}
\Big(\frac{1}{2H_p^2}\Big)_p+\omega(p)=0.
\end{equation}
Moreover, it satisfies the boundary conditions
\begin{equation}\label{M4ca}
\frac{1}{2H_p^2(1)}+H(1)=R,\;\;H(0)=0.
\end{equation}
The Froude number in new variables can be written as
$$
\frac{1}{F^2}=\int_0^1H_p^3dp.
$$

Furthermore, the Frechet derivative (the linear approximation of the functions ${\mathcal F}$ and ${\mathcal G}$ at a solution $h$) is the following
\begin{equation*}
Aw=A(h)w=\Big(\frac{h_qw_q}{h_p^2}-\frac{(1+h_q^2)w_p}{h_p^3}\Big)_p-\Big(\frac{w_q}{h_p}-\frac{h_qw_p}{h_p^2}\Big)_q
\end{equation*}
and
\begin{equation}\label{J4aba}
{\mathcal N}w={\mathcal N}(h)w=(N w-w)|_{p=1},
\end{equation}
where
\begin{equation*}
N w=N(h)w=\Big(-\frac{h_qw_q}{h_p^2}+\frac{(1+h_q^2)w_p}{h_p^3}\Big)\Big|_{p=1}.
\end{equation*}
The eigenvalue problem for the Frechet derivative, which is important for the analysis of bifurcations of the problem
(\ref{F19a}), is the following
\begin{eqnarray}\label{M1a}
&&A(h)w=\mu w\;\;\mbox{in $Q$},\nonumber\\
&&{\mathcal N}(h)w=0\;\;\mbox{for $p=1$},\nonumber\\
&&w=0\;\;\mbox{for $p=0$}.
\end{eqnarray}
Then according to  \cite{W} there exists a continuous branch of even solitary wave solutions to (\ref{J4a})
\begin{equation}\label{J4ac}
h=h(q,p;t):[0,\infty)\rightarrow C^{2,\gamma}_{0,e}(\overline{Q}),\;\;R=R(t):[0,\infty)\rightarrow (R_c,\infty),
\end{equation}
which  has a real analytic reparametrization locally around each $t> 0$.

An equivalent formulation of Theorem \ref{ThF18} in the variables $q$ and $p$ is given in the following

\begin{theorem}\label{TrF18b} Let condition (\ref{Okt30a}) hold. There exist at least one and at most $m$ continuous curves
$$
(\widehat{h}(s),t(s))\in C^{2,\alpha}_{0,e}(Q)\times (0,\infty),\;\;\;s\in [0,\infty),
$$
bifurcating from $(h(t_*),R(t_*)$. They admit an analytic re-parameterisation near each point, $t(0)=t_*$ and $(\widehat{h}(0),t(0))=(h(t_*),t_*)$. Moreover each curve is one of the following:

(i) $t(s)=t_*$ for all $s\geq 0$;

(ii) $\mu_1((t(s))=0$ for all $s\geq 0$;

(iii) $t'(s)\neq 0$   and the Frechet derivative is invertible  for $s>0$ except a set of discrete points with the only possible accumulation point at $\infty$.

Furthermore the bifurcation curve satisfies one of the properties (a)-(d) in Theorem \ref{ThF18}. %

\end{theorem}

In this theorem $\widehat{h}=h(t)+w$, where $w=0$ for $t=t_*$.
Instead of $(h,R)$ we use the pair $(w,t)$ in (\ref{J4a}), where $h$ and $R$ are replaced by $\widehat{h}$ and $R(t)$.

\section{Some properties of solitary waves}

In this section we discuss such properties of sets of solitary waves as uniform boundedness, uniform decay at $\infty$ and compactness. Another approach of proving compactness of sets of solitary waves can be found in \cite{W}, \cite{W3} and \cite{W2}.

Denote
$$
Q_a=[a,a+1]\times [0,1]\;\;\mbox{for $a\in\Bbb R$}.
$$
For $k=1,2,\ldots$ and $\alpha\in (0,1)$ we denote $C^{k,\alpha}(Q_a)$ and $C^{k,\alpha}(a,a+1)$ the H\"{o}lder spaces in $Q_a$ and $[a,a+1]$ respectively. The norms in these spaces we denote $||\cdot||_{C^{k,\alpha}(Q_a)}$ and $||\cdot||_{C^{k,\alpha}((a,a+1))}$.

Let $C^{k,\alpha}_0(Q)$ the space of even functions define on $Q$, vanishing for $p=0$, with the finite norm
$$
||u||_{C^{k,\alpha}_0(Q)}:=\sup_{a\in\Bbb R}||\cdot||_{C^{k,\alpha}(Q_a)}.
$$
The space  $C^{k,\alpha}(\Bbb R)$ consists of even functions on $\Bbb R$ with the finite norm
$$
||u||_{C^{k,\alpha}(\Bbb R)}:=\sup_{a\in\Bbb R}||\cdot||_{C^{k,\alpha}((a,a+1))}.
$$

 We put
\begin{equation}\label{Ap5a}
{\mathcal U}_\delta=\{(h,R)\in C^{2,\alpha}_{0}(Q)\times (R_c-\delta,\infty),\,:\,\delta<h_p<\delta^{-1}, h_q<\delta^{-1}\}.
\end{equation}
where $\delta$ is a positive number.
Let also
\begin{equation}\label{Ja11azz}
{\mathcal S}=\{(h,R)\in C^{2,\alpha}_{0}(Q)\times (R_c,\infty) \,:\,\mbox{$(h,R)$--a solitary wave solution to (\ref{J4a})}\}.
\end{equation}

\subsection{Uniform boundedness of solitary waves in ${\mathcal U}_\delta$}

Introduce the constant
$$
\omega_1=\max_{\tau\in [0,1]}|\omega(\tau)|+\max_{\tau\in [0,1]}|\omega'(\tau)|.
$$
The proof of the following lemma can be found in \cite{KLN17},  Proposition 2.
\begin{lemma}\label{LM20a} Let
\begin{equation}\label{F27az}
R_c<R_1<\infty,\;\;\Psi_Y\geq 0\;\;\mbox{in $D_\xi$ and}\;\;|\xi'(X)|\leq M\;\;\mbox{for $X\in\Bbb R$}.
\end{equation}
Then there exists $C=C(R_1,M,\omega_1)$ (this constant does not depend
on other parameters) such that if $(\Psi,\xi)\in C^1(\overline{D_\xi})\times W^{1,\infty }(\Bbb R)$ solves the problem (\ref{K2a}) with $R\in (0,R_1]$
then
\begin{equation}\label{M16a}
|\nabla\Psi(X,Y)|\leq C.
\end{equation}
\end{lemma}

The next assertion follows from the above lemma and the local estimates  presented and proved in Sect. 2, Chapter 10, \cite{LU}. We use the notation $B_{\rho}(X,Y)$ for the ball of radius $\rho$ and with the center at $(X,Y)$.

\begin{lemma}\label{LM20b} Let the inequalities (\ref{F27az}) be valid and let $(\Psi,\xi)\in C^{2,\alpha}(\overline{D_\xi})\times C^{2,\alpha}(\Bbb R)$  be a solution to (\ref{K2a}) with $R_c<R_1\leq R$. If
\begin{equation}\label{F27b}
\Psi_Y\geq \delta>0 \;\mbox{in $B_\rho(X,Y)$}
\end{equation}
then there exists constants $C_1$ and $\alpha_1\in (0,1)$  depending also on $M,\omega,\delta,\rho$ and $R_1$ such that
$$
||\Psi||_{C^{3,\alpha_1}(D_\xi\bigcap B_{\rho/2}((X,Y))}\leq C_1,\;\;||\xi||_{C^{3,\alpha_1}(S_\xi\bigcap B_{\rho/2}((X,Y))}\leq C_1.
$$
\end{lemma}

Next lemma is a version of the previous one in the variables $(q,p)$.

\begin{lemma}\label{M25a}  Let the inequalities (\ref{F27az}) be valid and let $h\in C^{2,\alpha}(\overline{Q})$ be a solution  to (\ref{J4a}) with $R_c<R_1\leq R_1$ subject to $\delta\leq h_p\leq\delta^{-1}$ in $Q\bigcap B_\rho(q,p)$.
Then there exist constants $C_1$ and $\alpha_1\in (0,1)$ depending on $\delta$, $M$, $\omega$, $\rho$ and $R_1$, such that $h\in C^{3,\alpha_1}(Q\bigcap B_{\rho/2}(q,p))$ and
$$
||h||_{C^{3,\alpha_1}(Q\bigcap B_{\rho/2}(q,p))}\leq C_1.
$$
\end{lemma}

\begin{corollary}\label{Ap9b} Let ${\mathcal U}_\delta$ and ${\mathcal S}$  be introduced by (\ref{Ap5a}) and (\ref{Ja11azz}) respectively. The set $\overline{{\mathcal U}_\delta\bigcap\mathcal S}$ is bounded in
$C^{3,\alpha_1}_0(Q)\times (R_c,\infty)$ for certain $\alpha_1\in (0,1)$.

\end{corollary}
\begin{proof}
Since $|\xi'(x)|=|\Psi_x/\Psi_y|$, the result follows from the boundedness of $R$ for solitary waves, see Proposition \ref{Pr26b}(iii) and from the estimate $|\xi'|\leq \sup_{{\mathcal U}_\delta}|h_q|\leq 1/\delta$.
\end{proof}

Let $(h,R)\in {\mathcal U}_\delta$ be a solitary wave solution to (\ref{J4a}). We represent it as
\begin{equation}\label{F6a}
h(q,p)=H(p)+w(q,p),
\end{equation}
where $h(q,p)\rightarrow H(p)$ as $p\to\pm\infty$ and $H$ is given by (\ref{M4c}) with ${\mathcal R}(\theta)=R$. Let
$$
W_\delta=\{(w,R)\;:\; (h,R)\in {\mathcal U}_\delta\bigcap {\mathcal S}\},
$$
where $h$ and $w$ are by (\ref{F6a}).
Corollary \ref{Ap9b} implies
\begin{equation}\label{Ap23a}
\sup_{(w,R)\in W_\delta}\sup_{a\in\Bbb R}||w||_{C^{3,\alpha_1}(Q_a)}<\infty .
\end{equation}
By Proposition 2.5, \cite{W} the set of  $w$ in representation  (\ref{F6a}) satisfies the following equidecay at infinity property:
\begin{equation}\label{Ap23aa}
\lim_{q\to\pm\infty}\sup_{(w,R)\in {\mathcal U}_\delta}\sup_{p\in[0,1]}|w(q,p)|=0.
\end{equation}


Making the substitution $h=H+w$ in  the system (\ref{J4a})
\begin{eqnarray}\label{Ap5aa}
&&\Big(-\frac{(2H_p+w_p)w_p}{2(H_p+w_p)^2H_p^2}+\frac{w_q^2}{2(H_p+w_p)^2}\Big)_p
-\Big(\frac{w_q}{H_p+w_p}\Big)_q=0
\;\;\mbox{in $Q$}\nonumber\\
&& -\frac{(2H_p+w_p)w_p}{2(H_p+w_p)^2H_p^2}+\frac{w_q^2}{2(H_p+w_p)^2}+w=0\;\;\mbox{for $p=1$},\nonumber\\
&&w(q,0)=0.
\end{eqnarray}
We transform the system to
\begin{eqnarray}\label{Ap6a}
&&\Big(-\frac{w_p}{H_p^3}+{\mathcal J}(H,w)\Big)_p
-\Big(\frac{w_q}{H_p}-{\mathcal I}(H,w)\Big)_q=0
\;\;\mbox{in $Q$}\nonumber\\
&& -\frac{w_p}{H_p^3}+{\mathcal J}(H,w)
+w=0\;\;\mbox{for $p=1$},\nonumber\\
&&w(q,0)=0,
\end{eqnarray}
where
$$
{\mathcal J}(H,w)=\frac{(3H_p+2w_p)w^2_p}{2(H_p+w_p)^2H_p^3}+\frac{w_q^2}{2(H_p+w_p)^2}
\;\;\mbox{and}\;\;{\mathcal I}(H,w)=\frac{w_qw_p}{H_p(H_p+w_p)}.
$$

\;\;$\alpha_2\in (0,\alpha_1)$

Using uniform boundedness (\ref{Ap23a}), and the equidecay property (\ref{Ap23aa}), we obtain
\begin{equation}\label{Ap23b}
||w||_{C^{3}(Q_a)}\leq C \varpi(a),
\end{equation}
 where $\varpi(a)\to 0$ when $|a|\to\infty$.
and the constant $C$ in (\ref{Ap23b}) depends on $\delta$. 



\subsection{Uniform exponential decay of functions in ${\mathcal U}_\delta$}

Here we assume that $R>R_c$ and $H$ is given by (\ref{M4c}) with $\theta>\theta_c$.
We start this section by  considering the model linear problem
\begin{eqnarray}\label{Ap6aa}
&&\Big(-\frac{w_p}{H_p^3}\Big)_p-\Big(\frac{w_q}{H_p}\Big)_q=f\;\;\mbox{in $Q$}\nonumber\\
&& -\frac{w_p}{H_p^3}+w=g\;\;\mbox{for $p=1$},\nonumber\\
&&w(q,0)=0.
\end{eqnarray}

In the study of solvability of this problem an important role plays the spectral problem
\begin{eqnarray}\label{Ap6ad}
&&\Big(-\frac{v_p}{H_p^3}\Big)_p-\lambda^2\frac{w}{H_p}=0\nonumber\\
&&-\frac{v_p}{H_p^3}+w=0\;\;\mbox{for $p=1$},\nonumber\\
&&v(q,0)=0.
\end{eqnarray}
From Proposition \ref{PrAp8a} and \cite{KLN14a} it follows that the operator in this spectral problem has the same eigenvalues as the operator in the problem (\ref{Ap8a}) and hence it is positive definite if $H$ is a supercritical laminar flow corresponding to $R>R_c$. We denote by $\nu_0=\nu_0(R)$ the smallest eigenvalue, which is positive and continuously depending on $R$. Let $\lambda_1=\lambda_1(R):=\sqrt{\nu_0}$. 

Denote $C^{k,\alpha}_{\textrm{loc}}(Q)$  and $C^{k,\alpha}_{\textrm{loc}}(\Bbb R)$ the sets of even functions $u$ on $Q$ and $\Bbb R$ respectively with finite semi-norms $||u||_{C^{2,\alpha}(Q_q)}$ and $||u||_{C^{2,\alpha}(q,q+1)}$, $q\in\Bbb R$.

For $\beta\in \Bbb R$ we introduce the weighted H\"older spaces $C^{k,\alpha}_\beta(Q)$ and $C^{k,\alpha}(\Bbb R)$ consisting of functions on $Q$ and $\Bbb R$ with finite norms
$$
||u||_{C^{k,\alpha}_\beta(Q)}:=\sup_{a\in\Bbb R}e^{-\beta a}||\cdot||_{C^{k,\alpha}(Q_a)}
$$
and
$$
||u||_{C^{k,\alpha}_\beta(\Bbb R)}:=\sup_{a\in\Bbb R}e^{-\beta a}||\cdot||_{C^{k,\alpha}((a,a+1))}
$$
respectively. Let also $C^{k,\alpha}_{0,\beta}(Q)$ be the subspace of functions in $C^{k,\alpha}_\beta(Q)$ vanishing for $p=0$.

In the next proposition we present a solvability result in the spirit of the book \cite{KM2}.

\begin{proposition}\label{PrAp17a} Let $\alpha\in (0,1)$. Let also $f\in C^{0,\alpha}_{\textrm{loc}}(Q)$ and $g\in C^{1,\alpha}_{\textrm{loc}}(\Bbb R)$ satisfies
\begin{equation}\label{Ap17a}
\int_{\Bbb R} e^{-\lambda_1|q|}(||f||_{C^{0,\alpha}(Q_q)}+||g||_{C^{1,\alpha}(q,q+1)})dq<\infty.
\end{equation}
Then the problem (\ref{Ap6aa}) has a unique solution $w\in C^{2,\alpha}_{\textrm{loc}}(Q)$ subject to
\begin{equation}\label{Ap17aa}
||w||_{C^{2,\alpha}(Q_q)}=o(e^{\lambda_1|q|})\;\;\mbox{as $|q|\to \infty$},
\end{equation}
 which  satisfies the estimate
\begin{equation}\label{Ap17ab}
||w||_{C^{2,\alpha}(Q_q)}\leq \frac{C}{\lambda_1}\int_{\Bbb R} e^{-\lambda_1|q-q'|}(||f||_{C^{0,\alpha}(Q_{q'})}+||g||_{C^{1,\alpha}(q',q'+1)})dq'.
\end{equation}
The constant $C$ here depends on $\delta:=\min H_p$, on the norm $||H||_{C^{2,\alpha}([0,1])}$.
\end{proposition}

\begin{proof} We assume first that $g=0$ in (\ref{Ap6aa}).  Denote by $\lambda_j^2$, $j=1,\ldots$,  the eigenvalues of the problem (\ref{Ap6ad}), ordered according $0<\lambda_1<\lambda_2<\cdots$, and by $v_j$ corresponding eigenfunctions, which are normalized by
\[
\int_0^1v_jv_k\frac{dp}{H_p}=\delta_{j,k},
\]
where $\delta_{j,k}$ is  the Kronecker delta. We are looking for the solution to the problem (\ref{Ap6aa}) in the form
$$
w(q,p)=\sum_{j=1}^\infty c_j(q)v_j(p).
$$
Then $c_j$ satisfies
$$
\lambda_j^2c_j-c^{''}_j=f_j,
$$
where $f_j$ are coefficients in the decomposition
$$
H_pf=\sum_j^\infty f_jv_j.
$$
If $|c_j(q)|=o(e^{\lambda_1|q|})$ then the solution $c_j$  is given by
$$
c_j=\frac{1}{2\lambda_j}\int_{\Bbb R}e^{-\lambda_j|q-q'|}f_j(q')dq'.
$$
Using Minkowski's integral inequality, we get
$$
||w||_{L^2(0,1)}\leq C\int_{\Bbb R}e^{-\lambda_1|q-q'|}||f(q',\cdot)||_{L^2(0,1)}dq',
$$
where $C$ depends on $\delta$. Again applying Minkowski's inequality, we obtain (see \cite{KM2})
$$
||w||_{L^2(Q_t)}\leq C\int_{\Bbb R}e^{-\lambda_1|t-s|}||f||_{L^2(Q_s)}ds.
$$
Using local estimates for the problem (\ref{Ap6aa}), we arrive at
$$
||w||_{C^{2,\alpha}(Q_t)}\leq C\int_{\Bbb R}e^{-\lambda_1|t-s|}||f||_{C^{0,\alpha}(Q_s)}ds.
$$

If $g$ is not identically zero then we can find $W\in C^{2,\alpha}_{\textrm{loc}}(Q)$ satisfying $-\frac{W_p}{H_p^3}+W=g$ and  the inequality
$$
||W||_{C^{2,\alpha}(Q_t)}\leq C||g||_{C^{1,\alpha}(t-1,t+2)}.
$$
This reduce the problem to a the case $g=0$ as a final result we get the inequality
$$
||w||_{C^{2,\alpha}(Q_t)}\leq C\int_{\Bbb R} e^{-\lambda_1|t-s|}(||f||_{C^{0,\alpha}(Q_s)}+||g||_{C^{1,\alpha}(t,t+1)})ds.
$$
This estimate implies the required inequalities.
\end{proof}

As a corollary we give a more usual version of the solvability result

\begin{corollary}\label{PrAp12a} Let $\alpha\in (0,1)$ and $\beta\in (-\lambda_1,\lambda_1)$. Let also $f\in C^{0,\alpha}_\beta(Q)$ and $g\in C^{1,\alpha}_\beta(\Bbb R)$. Then the problem (\ref{Ap6aa}) has a unique solution $w\in C^{2,\alpha}_{0,\beta}(Q)$ which  satisfies the estimate
\begin{equation}\label{Ap6bc}
||w||_{C^{2,\alpha}_\beta(Q)}\leq \frac{C}{\lambda_1}\Big(\frac{1}{|\beta-\lambda_1|}+\frac{1}{|\beta+\lambda_1|}\Big)
(||f||_{C^{0,\alpha}_\beta(Q)}+||g||_{C^{k,\alpha}_\beta(\Bbb R)}),
\end{equation}
where $C$ depends on $\delta:=\min H_p$ and on the norm $||H||_{C^{2,\alpha}([0,1])}$.

If additionally  $f\in C^{0,\alpha}_{\beta'}(Q)$ and $g\in C^{1,\alpha}_{\beta'}(\Bbb R)$ with $\beta'\in (-\lambda_1,\lambda_1)$ then solutions corresponding to $\beta$ and $\beta'$ coincides.
\end{corollary}

Introduce
$$
\widehat{\lambda}_1(\delta)=\sup\lambda_1(R)
$$
where supremum is taken with respect to all supercritical laminar flows $(U,d,R)$ serving as asymptotics at $\pm\infty$ for  elements from ${\mathcal U}_\delta\bigcap {\mathcal S}$.

\begin{theorem}\label{TAp25} Let $\delta>0$ and $\varepsilon>0$ and let $\alpha\in (0,1)$. Then there exist $\widehat{q}>0$ depending on $\delta$, $\varepsilon$ and $\alpha$ such that all solitary waves $h$ in $W_\delta\bigcap {\mathcal S}$ satisfy the estimate
\begin{equation}\label{Ap6da}
||w||_{C^{2,\alpha}(Q_q)}\leq Ce^{-(\widehat{\lambda}_1(\delta)-\varepsilon)|q|}\;\;\mbox{for $|q|>\widehat{q}$},
\end{equation}
where $C=C(\delta,\alpha,\varepsilon)$ depends on $\delta$, $\alpha$ and $\varepsilon$.
\end{theorem}
\begin{proof}
Let $\zeta$ be a smooth function on $\Bbb R$ such that $\zeta(q)=1$ for $q>1$ and $\zeta(q)=0$ for $q<0$. Introduce the function
$$
w_{\widehat{q}}(q,p)=\zeta(q-\widehat{q})w(q,p).
$$
Then
\begin{eqnarray}\label{Ap11a}
&&\Big(-\frac{(w_{\widehat{q}})_p}{H_p^3}+\zeta(q-\widehat{q}+1){\mathcal J}(H,w_{\widehat{q}})\Big)_p
-\Big(\frac{(w_{\widehat{q}})_q}{H_p}-\zeta(q-\widehat{q}+1){\mathcal I}(H,w_{\widehat{q}})\Big)_q=f_{\widehat{q}}
\;\;\mbox{in $Q$}\nonumber\\
&& (-\frac{(w_{\widehat{q}})_p}{H_p^3}+\zeta(q-\widehat{q}+1){\mathcal J}(H,(w_{\widehat{q}})_p)
+w_{\widehat{q}}=g_{\widehat{q}}\;\;\mbox{for $p=1$},\nonumber\\
&&w(q,0)=0,
\end{eqnarray}
where the functions $f_{\widehat{q}}$ and $g_{\widehat{q}}$ vanish outside the interval $(\widehat{q},\widehat{q}+1)$.

Introduce the operators
$$
Aw=\Big(-\frac{w_p}{H_p^3}\Big)_p-\Big(\frac{w_q}{H_p}\Big)_q
$$
and
$$
Bw=\Big(-\frac{w_p}{H_p^3}+w\Big)\Big|_{p=1}.
$$
Let also
$$
F_{\widehat{q}}(w)=\Big(\zeta(q-\widehat{q}+1){\mathcal J}(H,w)\Big)_p
+\Big(\zeta(q-\widehat{q}+1){\mathcal I}(H,w)\Big)_q
$$
and
$$
G_{\widehat{q}}(w)=\zeta(q-\widehat{q}+1){\mathcal J}(H,(w)_p
$$
Then due to (\ref{Ap23b})
$$
||F_{\widehat{q}}(w_{\widehat{q}})||_{C^{0,\alpha}(Q_q)}\leq \varpi(\widehat{q})||w_{\widehat{q}}||_{C^{0,\alpha}(Q_q)}
$$
and
$$
||G_{\widehat{q}}(w_{\widehat{q}})||_{C^{0,\alpha}(q,q+1)}\leq \varpi(\widehat{q})||w_{\widehat{q}}||_{C^{0,\alpha}(Q_q)},
$$
where $\varpi(\widehat{q})$ tends to zero when $\widehat{q}$ tends to $\infty$. Applying Proposition \ref{PrAp17a}
\begin{eqnarray}\label{Ap24a}
&&||w_{\widehat{q}}||_{C^{2,\alpha}(Q_q)}\leq C\varpi(\widehat{q})\int_{-\infty}^\infty e^{-\widehat{\lambda}_1|q-q'|}||w_{\widehat{q}}||_{C^{2,\alpha}(Q_{q'})}dq'\nonumber\\
&&+Ce^{-\widehat{\lambda}_1|q-\widehat{q}|}||w||_{C^{2,\alpha}(Q_{\widehat{q}})}.
\end{eqnarray}
To simplify this estimate we put
$$
v(q)=\frac{1}{2\widehat{\lambda}_1}\int_{-\infty}^\infty e^{-\widehat{\lambda}_1|q-q'|}||w_{\widehat{q}}||_{C^{2,\alpha}(Q_{q'})}dq'.
$$
Then
\begin{eqnarray*}
&&-v^{''}+\widehat{\lambda}^2_1v=||w_{\widehat{q}}||_{C^{2,\alpha}(Q_{q})}\leq C\varpi(\widehat{q})v\\
&&+Ce^{-\widehat{\lambda}_1|q-\widehat{q}|}||w||_{C^{2,\alpha}(Q_{\widehat{q}})}.
\end{eqnarray*}
Let $\mu>$ be defined by $\widehat{\mu}^2=\widehat{\lambda}^2_1-C\varpi(\widehat{q})$.  Then
$$
v(q)\leq\frac{C}{2\widehat{\mu}}\int_{-\infty}^\infty e^{-\widehat{\mu}|q-q'|}
e^{-\widehat{\lambda}_1|q'-\widehat{q}|}dq'||w||_{C^{2,\alpha}(Q_{\widehat{q}})}.
$$
This together with (\ref{Ap24a}) implies (\ref{Ap6da}).

\end{proof}

The following proposition is very important in forthcoming bifurcation analysis.
\begin{proposition}\label{PrF12az} Let ${\mathcal U}_\delta$ and ${\mathcal S}$  be introduced by (\ref{Ap5a}) and (\ref{Ja11azz}) respectively. The set $\overline{{\mathcal U}_\delta\bigcap\mathcal S}$ is compact in
$C^{2,\alpha}_0(Q)\times (R_c,\infty)$.

\end{proposition}
\begin{proof} By Corollary \ref{Ap9b} the set $\overline{{\mathcal U}_\delta\bigcap {\mathcal S}}$ is bounded in $C^{3,\alpha_1}_0(Q)$ for certain $\alpha_1\in (0,1)$. Therefore its restriction onto  $[-N,N]\times [0,1]$ is compact in $C^{2,\alpha}([-N,N]\times [0,1])$ for any $N>0$.

From Theorem \ref{TAp25} it follows that for every $\epsilon_1>0$ there exists $\widehat{q}$ such that
$$
||h||_{C^{2,\alpha}(\widehat{q},\infty)\times [0,1])}+||h||_{C^{2,\alpha}(-\infty,\widehat{q})\times [0,1])}
\leq \epsilon_1
$$
for all $h\in \overline{{\mathcal U}_\delta\bigcap{\mathcal S}}$.
These two properties imply the required compactness.


\end{proof}

\section{Abstract bifurcation analysis}\label{Smain}

Here we present some known results about bifurcations in Banach spaces, which are taken from \cite{Ki1}, \cite{Ki2}, \cite{BTT} and \cite{CSrVar}.

Let ${\bf X}$, ${\bf Y}$ and ${\bf Z}$ be real Banach spaces, ${\bf X}\subset {\bf Z}$ and ${\bf X}$ continuously embedded into ${\bf Z}$. Consider a function
$$
{\mathcal F}=(F,G)=(F(x,\lambda),G(x,\lambda))\;\;\mbox{ mapping}\;\; {\mathcal U}\rightarrow  {\bf Z}\times {\bf Y},
$$
where ${\mathcal U}$ is an open set in ${\bf X}\times\Bbb R$ and $F$ and $G$ are analytic functions in ${\mathcal U}$. We assume that $(0,0)\in {\mathcal U}$ and
\begin{equation}\label{01}
{\mathcal F}(0,\lambda)=0\;\;\mbox{for all $(0,\lambda)\in {\mathcal U}$}.
\end{equation}
We are interested in solutions $(x,\lambda)$ to the problem\footnote{This formulation is different from that in  \cite{Ki1} and \cite{Ki2} because of the additional Banach space $Y$. The advantage of this setting is that it allows applications to boundary value problems and the forthcoming analysis does not undergo any changes.}
\begin{equation}\label{1}
{\mathcal F}(x,\lambda)=0
\end{equation}
with  $x\neq 0$.

Let $\delta>0$ and
$$
B_\delta=\{x\in {\bf X}\,:\,||x||_{\bf X}<\delta\}\;\;\;\mbox{and}\;\;I_\delta=\{\lambda\in \Bbb R\,:\,|\lambda|<\delta\}.
$$
For sufficiently small $\delta$ the set $B_\delta\times I_\delta$ belongs to ${\mathcal U}$ and the analytic functions  $F$ and $G$ can be represented as
\begin{equation}\label{2}
F(x,\lambda)=\sum_{j\geq 0,k\geq 0} \lambda^jF_{jk}(x)\;\;\mbox{and}\;\;G(x,\lambda)=\sum_{j\geq 0,k\geq 0}^\infty \lambda^jG_{jk}(x),
\end{equation}
where $F_{jk}:{\bf X}^k\rightarrow {\bf Z}$ and $G_{jk}:{\bf X}^k\rightarrow {\bf Y}$ are $k$--linear, symmetric and continuous operator, and the above series are convergent in $B_\delta\times I_\delta$. Due to (\ref{01}), $F_{j0}=0$ and $G_{j0}=0$
for all $j\geq 0$. We represent $F$ and $G$ as
\begin{equation}\label{F1a}
{\mathcal F}(x,\lambda)=A(\lambda)x+\widehat{{\mathcal F}}(x,\lambda),
\end{equation}
where
\begin{equation}\label{D4a}
A(\lambda)=\sum_{j=0}^\infty \lambda^jA_j\,:\, X\to Z\times Y,\;\;A_jx=(F_{j1}(x),G_{j1}(x))
\end{equation}
and
$$
\widehat{{\mathcal F}}(x,\lambda)=(\widehat{F}(x,\lambda),\widehat{G}(x,\lambda)),\;\;\widehat{F}(x,\lambda)=\sum_{j\geq 0,k\geq 2} \lambda^jF_{jk}(x),\;\;\widehat{G}(x,\lambda)=\sum_{j\geq 0,k\geq 2}^\infty \lambda^jG_{jk}(x).
$$

We assume  that the Frechet derivatives
$$
(D_xF(x,\lambda)v,D_xG(x,\lambda)v)\,:\,{\bf X}\rightarrow {\bf Z}\times {\bf Y}
$$
is a Fredholm operator of index zero for all $(x,\lambda)\in {\mathcal U}$ satisfying (\ref{1}).

The Frechet derivative at $x=0$ is given by
$$
(D_xF(0,\lambda)v,D_xG(0,\lambda)v)=A(\lambda)v.
$$

We will use the following spectral problem in order to study the bifurcation of (\ref{1}) at $(x,\lambda)=(0,0)$:
\begin{equation}\label{D3a}
A(\lambda)v=\mu(\lambda)(v,0).
\end{equation}

 We assume also that for $\lambda=0$ the problem (\ref{D3a})
has the eigenvalue $\mu=0$, which is algebraically simple. Therefore there is a function $\mu(\lambda)$ of eigenvalues of  (\ref{D3a}) consisting of simple eigenvalues of (\ref{D3a}). This function is analytic in a neighborhood of $\lambda=0$ and there are no other eigenvalues in this neighborhood. There is also an analytic function $v(\lambda)$ consisting of eigenfunctions of the (\ref{D3a}) corresponding to the eigenvalue $\mu(\lambda)$. We write the function $\mu$ as
\begin{equation}\label{3}
\mu(\lambda)=\sum_{j=m}^\infty\hat{\mu}_j\lambda^j.
\end{equation}
Our main assumption is
\begin{equation}\label{4}
m\;\;\mbox{is odd and}\;\;\hat{\mu}_m\neq 0.
\end{equation}

It is reasonable to study  small bifurcation curves of (\ref{1})  first and then to discuss their  continuation to curves containing large bifurcating solutions.

\subsection{Small bifurcation curves}\label{SD1}

Since  $v(\lambda)$ is analytic function, it can be represented as
$$
v(\lambda)=\sum_{j=0}^\infty v_j\lambda^j.
$$
 Furthermore, the spectral problem $A(\lambda)^*(\tilde{z},\tilde{y})=\tilde{\mu}\tilde{z}$ has the eigenvalue $\tilde{\mu}=\mu(\lambda)$ which is also simple and we denote the corresponding eigenfunction by
 $\tilde{w}(\lambda)=(\tilde{z}(\lambda),\tilde{y}(\lambda))\in {\bf Z}^*\times {\bf Y}^*$. They can be chosen  such that the functions $\tilde{z}(\lambda)$ and $\tilde{y}(\lambda)$ are analytic in $\lambda$ and,
 due to algebraic simplicity of the eigenvalue $\mu(\lambda)$, satisfy
 $$
 \langle v(\lambda),\tilde{z}(\lambda)\rangle\neq 0\;\;\mbox{for small $\lambda$}.
 $$
 The coefficients $v_j$ can be chosen to satisfy
$$
\langle v_j,\tilde{z}(0)\rangle=0\;\;\mbox{for $j=1,\ldots$}.
$$

 We put
 $$
 \widehat{v}(\lambda)=v(\lambda)\;\;\mbox{and}\;\;\widehat{w}(\lambda)=(\widehat{z},\widehat{y})=\frac{\tilde{w}(\lambda)}{\langle v(\lambda),\tilde{z}(\lambda)\rangle}.
 $$
 Both these functions are still analytic and satisfy
  $$
 \langle\widehat{v}(\lambda),\widehat{z}(\lambda)\rangle=1\;\;\mbox{for small $\lambda$}.
 $$

In order to reduce the system (\ref{1}) to an one dimensional equation one can use the Lyapunov-Schmidt method. We define the following projections
$$
P_\lambda x=\langle x,\widehat{z}(\lambda)\rangle \widehat{v}(\lambda),\;\;\widehat{P}_\lambda(z,y)=\langle (x,y),\widehat{w}(\lambda)\rangle (\widehat{v}(\lambda),0).
$$
Here and in what follows we identify the element $w\in {\bf X}$ with the element $(w,0)\in {\bf Z}\times {\bf Y}$ in the definition of $\widehat{P}_\lambda$. Clearly $P_\lambda x=\widehat{P}_\lambda(x,0)$. Now we can write (\ref{1}) as
\begin{eqnarray}\label{D5a}
&&\widehat{P}_\lambda{\mathcal F}(v+w,\lambda)=0,\;\;x=P_\lambda x+(I-P_\lambda)x=v+w,\nonumber\\
&&(I-\widehat{P}_\lambda){\mathcal F}(v+w,\lambda)=0.
\end{eqnarray}
Taking $v=t\widehat{v}(\lambda)$ and using representation (\ref{F1a}), we get
\begin{equation}\label{F1b}
t\mu(\lambda)+\langle \widehat{{\mathcal F}}(t\widehat{v}+w,\lambda),\widehat{w}(\lambda)\rangle=0
\end{equation}
and
\begin{equation}\label{F1ba}
(I-\widehat{P}_\lambda)(A(\lambda)w+\widehat{{\mathcal F}}(t\widehat{v}+w,\lambda))=0.
\end{equation}
We write the second equation in the form
\begin{equation}\label{F1bb}
(I-\widehat{P}_\lambda)\big(A(\lambda)w+\widehat{{\mathcal F}}_1(t\widehat{v}+w,\lambda)+\widehat{{\mathcal F}}(t\widehat{v},\lambda)\big)=0,
\end{equation}
where
\begin{equation}\label{F1bc}
\widehat{{\mathcal F}}_1(t\widehat{v},w,\lambda))=\widehat{{\mathcal F}}(t\widehat{v}+w,\lambda)-\widehat{{\mathcal F}}(t\widehat{v},\lambda).
\end{equation}
Since the operator
$$
A(\lambda):(I-P_\lambda){\bf X}\rightarrow (I-\widehat{P}_\lambda)({\bf Z}\times {\bf Y})
$$
is invertible by
using the implicit function theorem one can find $w=w(t,\lambda)$ which is analytically depends on $t$ and $\lambda$ and one can verify that
$$
w(t,\lambda)=t^2\overline{w}(t,\lambda),
$$
where the function $\overline{w}(t,\lambda)$ is also analytic in $t$ and $\lambda$. Now the equation (\ref{F1b}) can be written as
\begin{equation}\label{F1c}
\mu(\lambda)+t^{-1}\langle \widehat{{\mathcal F}}(t\widehat{v}+t^2\overline{w}(t,\lambda),\lambda),\widehat{w}(\lambda)\rangle=0
\end{equation}
Thus we arrive at the equation
\begin{equation}\label{F2a}
\mu(\lambda)+t{\mathcal T}(t,\lambda)=0,
\end{equation}
where ${\mathcal T}$ is an analytic function.

The main local bifurcation result for the equation (\ref{F2a}), and hence for the equation (\ref{1}), is presented in the next theorem, proved in \cite{Ki1}, Theorem I.16.4.

\begin{theorem}\label{ThJ3} Let $\mu(\lambda)$ satisfy {\rm (\ref{3})} and {\rm (\ref{4})}. Then bifurcating solutions to {\rm (\ref{1})} consist of two groups of the same number of curves   bifurcating at $(x,\lambda)=(0,0)$, and this number is at least one and at most $m$. 
One group consists of
\begin{eqnarray}\label{Dec28a}
&&x(s)=s\widehat{v}(\lambda)+\psi(s\widehat{v}(\lambda),\lambda(s))=s\widehat{v}(\lambda)+o(|s|),\nonumber\\
&&\lambda(s)=s^\gamma\sum_{k=0}^\infty a_ks^{k/p\gamma_2}\;\;\mbox{for $0\leq s<\delta$}
\end{eqnarray}
and in another  group the formulae for $x$ is the same but
\begin{equation}\label{Dec28aa}
\lambda(s)=(-s)^{\tilde{\gamma}}\sum_{k=0}^\infty \tilde{a}_k(-s)^{k/\tilde{p}\tilde{\gamma}_2}\;\;\mbox{for $-\delta< s\leq 0$}.
\end{equation}
Here  $\gamma=\gamma_1/\gamma_2$, $\tilde{\gamma}=\tilde{\gamma}_1/\tilde{\gamma}_2$ and $\gamma_1$, $\gamma_2$, $\tilde{\gamma}_1$, $\tilde{\gamma}_2$ and $p$, $\tilde{p}$ are positive integers.
The function $\psi$ is analytic near the point $(0,0)$ and  satisfies
$$
\langle \psi,\widehat{z}_0'\rangle=0.
$$
 We include here possible vertical bifurcation $\lambda(s)\equiv 0$. If $\lambda(s)$ is not the vertical bifurcation then $a_0\neq 0$, $\tilde{a}_0\neq 0$.
\end{theorem}

It follows from \cite{Ki1}, \cite{Ki2} the next

\begin{proposition}\label{PrDa} Let $M$ be the number of branches in each group of branches in the above theorem.  We numerate the branches  (\ref{Dec28a}) in increasing order if $M>1$, namely
$\lambda_1(s)<\cdots <\lambda_M(s)$ for $s\in (0,\delta)$. 
Then for each branch $(\lambda_j,x_j)$ there exists a unique branch $(\lambda^{(-)}_j,x^{(-)}_j)$ from the second group (\ref{Dec28aa}) in the above theorem  such that the function
\begin{equation}\label{F19a}
g_j(s)=(\lambda_j(s),x_j(s))\;\;\mbox{for $s\geq 0$ and}\;\;g_j(s)=(\lambda^{(-)}_j(s),x^{(-)}_j(s))\;\;\mbox{for $s<0$}
\end{equation}
admits an analytic injective reparametrization. If $\lambda_j$ has the form (\ref{Dec28a}) then the parametrization can be chosen as  $t=s^{1/p\gamma_2}$.

Moreover the following alternatives are valid for the above branches

\smallskip
(i) If $\lambda_j$ is not identically zero then $\lambda'(s)\neq 0$ for small $s\neq 0$ and

 $D_x (F,G)(x_j(s),\lambda_j(s))$ is invertible for small $s\neq 0$,

\smallskip
 (ii) $\lambda'(s)\neq 0$ for small $s\neq 0$ and $\mu_1(s)=0$ for all small $s$.

\smallskip
(iii) If $\lambda_j\equiv 0$ then  the corresponding solution is given by
$$
\lambda=0\;\;\;x(s)=s\widehat{v}_0+\psi(s\tilde{v}_0,0).
$$
\end{proposition}

\begin{remark} The conditions (\ref{3}) and (\ref{4}) can be relaxed as follows. We can assume that the function in the left hand side of the bifurcation equation (\ref{F2a}) is not identically zero. In this case Theorem \ref{ThJ3} and Proposition \ref{PrDa} are still true except that the number of branches can be zero, but if we have a branch for small $s>0$ then there is also a branch for small negative $s$ and the branch defined now for small $|s|$ admits an analytic re-parametrization, see \cite{Ki1}, Sect. 1.15, 1.16.

\end{remark}

\subsection{Global bifurcation, continuation of local bifurcation curve.}\label{SectJ3}

We denote by ${\mathcal B}^+$ the branch of small bifurcating solutions constructed in the previous section, see (\ref{F19a}). We denote one of them by
$(\kappa(s),\Lambda(s))$, which is given for $0<s<\epsilon$. By Proposition \ref{PrDa} there exists at least one such curve and it can be extended in a neighborhood of $0$ and the extended curve admits an analytic reparametrization. But in what follows we will use only the branch defined for $s\in (0,\epsilon)$.

 We assume that

(A) ${\mathcal F}(0,\lambda)=0$ for all $(0,\lambda)\in {\mathcal U}$.

(B) $D_x{\mathcal F}[(x,\lambda)]$ is a Fredholm operator of index zero when ${\mathcal F}(x,\lambda)=0$, $(x,\lambda)\in {\mathcal U}$.

(C) There is an interval $I\in\Bbb R$, in our case $I=(0,\epsilon)$ but having in mind a more general application of the next extension theorem we give a more general presentation, and an analytic function $(x,\lambda)=(\kappa(s),\Lambda(s))$, $s\in I$, $(\kappa(s),\Lambda(s))\in {\mathcal U}$, such that ${\mathcal F}(\kappa(s),\Lambda(s)))=0$ and $\Lambda'(s)\neq 0$ for $\lambda\in I$. Furthermore the operator $D_x{\mathcal F}[(\kappa(s),\Lambda(s))]$ is invertible for  $s\in I$.

Introduce
\begin{eqnarray}\label{D5b}
&&{\mathcal S}=\{(x,\lambda)\in{\mathcal U}\,:\, {\mathcal F}(x,\lambda)=0\},\nonumber\\
&&{\mathcal B}^+=\{(\kappa(s),\Lambda(s))\,:\,s\in I\}.
\end{eqnarray}

Usually in formulations of the global bifurcation theorem one starts from a local analytic branch coming from the analytic version  of Crandall, Rabinowitz Theorem, see \cite{CrRab}. Actually,  it is sufficient to replace this by a local analytic curve satisfying (A)-(C) and then one can use the same argument concerning extension of the local analytic curve to a global analytic one used in  \cite{BTT}.
We use the following version of the global bifurcation theorem taken from \cite{CSrVar}, where some clarifications for the main version of the bifurcation theorem from \cite{BTT} can be found.

\begin{theorem}\label{ThJ3a} Suppose {\rm (A)-(C)} hold and for some sequence ${\mathcal K}_j$, $j\in \mathbb{N}$, of bounded closed subsets of ${\mathcal U}$ with ${\mathcal U}=\bigcup_{j\in\mathbb{N}}{\mathcal K}_j$, the set ${\mathcal S}\bigcap {\mathcal K}_j$ is compact for each $j\in\mathbb{N}$.
Then there exists a continuous curve- ${\mathcal B}$, which extends ${\mathcal B}^+$ as follows

\medskip
(a) ${\mathcal B}=\{(\kappa(s),\Lambda(s))\,:\, s\in \Bbb R\}\subset {\mathcal U}$,\\
where $(\kappa,\Lambda): \Bbb R\rightarrow {\bf X}\times \Bbb R$ is continuous;

\medskip
(b) ${\mathcal B}^+\subset {\mathcal B}\subset {\mathcal S}$;

\medskip
(c) ${\mathcal B}$ has a real-analytic reparametrization locally around each of its points;



\medskip
(d) One of the following alternatives occurs:

$(\alpha)$ for every  $j\in \mathbb{N}$, there exists $s_j>0$ such that $(\kappa(s);\Lambda(s))$ does not belong to ${\mathcal K}_j$ for all $s>0$ with $s>s_j$;

$(\beta)$ there exists $T>0$ such that $(\kappa(s+T),\Lambda(s+T))=(\kappa(s),\Lambda(s))$ for all $s>0$.

Moreover, such a curve of solutions to ${\mathcal F}(x,\lambda)=0$ having the properties (a)-(d) is unique (up to reparametrization).





\end{theorem}


\section{The first bifurcation on the branch of solitary waves}

\subsection{Bifurcation equation}

Let $C^{k,\alpha}(Q)$ and $C^{k,\alpha}(\Bbb R)$ be usual H\"{o}lder spaces in $\overline{Q}$ and $\Bbb R$ respectively and let $C^{k,\alpha}_{0,e}(Q)$ be the subspace of $C^{k,\alpha}(Q)$ consisting of even functions $h$ vanishing for $p=0$. The space $C^{k,\alpha}_e(\Bbb R)$ consists of even functions in $C^{k,\alpha}(\Bbb R)$.  We put
$$
{\bf X}=C^{2,\alpha}_{0,e}(Q),\;\;{\bf Z}=C^{0,\alpha}_e(Q),\;\;{\bf Y}=C^{2,\alpha}_e(\Bbb R).
$$
and
\begin{equation}\label{Ja11a}
{\mathcal U}=\bigcup_\delta{\mathcal U}_\delta,
\end{equation}
where ${\mathcal U}_\delta$ was introduced by (\ref{Ap5a}).

Assuming that $h=h(q,p;t)$ and $R=R(t)$ is an element of the branch (\ref{J4ac}), we are looking for a solution to  the problem (\ref{J4a}) in the form $\widehat{h}=h+w$. Then for the function $w$ we get the following equations
\begin{eqnarray*}
&&F(w;t):=\Big(\frac{2h_p^3w_q-(1+h_q^2)(2h_pw_p+w_p^2)+h_p^2w_q^2}{2h_p^2(h_p+w_p)^2}\Big)_p\\
&&-\Big(\frac{h_pw_q-h_qw_p}{h_p(h_p+w_p)}\Big)_q=0\;\;\mbox{in $Q$}\\
&&G(w;t):=\frac{2h_p^3w_q-(1+h_q^2)(2h_pw_p+w_p^2)+h_p^2w_q^2}{2h_p^2(h_p+w_p)^2}-w=0\;\;\mbox{for $p=1$},\\
&&w(q,0;t)=0.
\end{eqnarray*}
Now this problem has a parameter $t>0$ and we are looking for the function $w\in {\bf X}$.

Let us introduce
\begin{equation}\label{M7a}
\widehat{\mathcal U}_\delta =\{(w,t)\,:\,(\widehat{h}, R(t))\in {\mathcal U}_\delta\,, t>0\},\;+;\widehat{h}(q,p;t)=h(q,p;t)+w,
\end{equation}
$$
\widehat{\mathcal U}=\bigcup_\delta \widehat{{\mathcal U}}_\delta
$$
and
\begin{equation}\label{M7az}
\widehat{{\mathcal S}}=\{(w,t)\,:\,(\widehat{h}, R(t))\in {\mathcal S} \,,t>0\}.
\end{equation}
The above sets are convenient for study of global bifurcation branches through the bifurcation point $(h(t_*),R(t_*))$. Thus, $(w,t)\in {\mathcal U}_\delta$ is equivalent to $(\widehat{h}(t),R(t))\in \widehat{\mathcal U}_\delta$.
Hence
\begin{equation}\label{Ja11aa}
{\mathcal F}=(F,G)\,:\, \widehat{{\mathcal U}}\rightarrow {\bf Z}\times {\bf Y}.
\end{equation}
Using this definition, we see that
\begin{equation}\label{Ja11b}
F(0;t)=0\;\;\;\mbox{and}\;\;\;G(0;t)=0.
\end{equation}
Suppose  that the property (\ref{Okt30a}) is satisfied, we put $\lambda=t-t_*$, in order to keep the same notation as in the previous sections.
Our aim is to study bifurcations of the problem
\begin{equation}\label{Ja11ba}
{\mathcal F}(w;t_*+\lambda)=(F(w;t_*+\lambda),G(w;t_*+\lambda))=0
\end{equation}
at $\lambda=0$ and $w=0$.

The Frechet derivative of the pair $(F,G)$ is the following
\begin{equation}\label{J4aa}
Av=A(w,\lambda)v=\Big(\frac{\lambda^2(h+w)_qv_q}{(h+w)_p^2}-\frac{(1+(h+w)_q^2)v_p}{(h+w)_p^3}\Big)_p-\Big(\frac{v_q}{(h+w)_p}-\frac{(h+w)_qv_p}{(h+w)_p^2}\Big)_q
\end{equation}
and
\begin{equation*}
{\mathcal N}v={\mathcal N}(w,\lambda)v=(N v-v)|_{p=1},
\end{equation*}
where
\begin{equation*}
N v=N(w,\lambda)v=\Big(-\frac{(h+w)_qv_q}{(h+w)_p^2}+\frac{(1+(h+w)_q^2)v_p}{(h+w)_p^3}\Big)\Big|_{p=1}.
\end{equation*}
The eigenvalue problem for the Frechet derivative, which is important for the analysis of bifurcations of the problem
(\ref{Ja11ba}), is the following
\begin{eqnarray*}
&&Av=\mu v\;\;\mbox{in $Q$},\\
&&{\mathcal N}v=0\;\;\mbox{for $p=1$},\\
&&v=0\;\;\mbox{for $p=0$}.
\end{eqnarray*}
Here $Q$ is the same strip as in Sect. \ref{SJ29a}.

By (\ref{Okt30a}) and Theorem \ref{T1a} $\mu_1(t_*)=0$ and the eigenvalue $\mu_1(t)$ is simple and admits the representation (\ref{F18a}). We denote by $\widehat{v}=\widehat{v}(t)$ the corresponding eigenfunction, which is normalized by
$$
\int_{Q}\widehat{v}^2dqdp=1.
$$
The the projection operator, introduced in Sect. \ref{SD1}, can be written as
\begin{equation}\label{M24a}
P_0w=\int_{Q}w\widehat{v}dqdp\,\widehat{v},\;\;\widehat{P}_0(x,y)=\Big(\int_{Q}x\widehat{v}dqdp
+\int_{\Bbb R}y\widehat{v}(q,0)dq\Big) (\widehat{v},0)
\end{equation}


\subsection{Lemmas}

In this section we prove Proposition \ref{Pr26} (see Remark \ref{R21a}) and  present some mostly known results which will be used in the proofs of  Theorems \ref{TrF18b} and \ref{ThF18}.

\begin{lemma}\label{LM28a} Let $R>R_c$ and $(\Psi,\xi)\in C^{2,\alpha}(D_\xi)\times C^{2,\alpha}(\Bbb R)$ satisfy (\ref{K2a}).Then the following properties hold:

(i) If
\begin{equation}\label{Au18av}
|\xi'(X)|\leq M\;\;\mbox{for $X\in\Bbb R$}
\end{equation}
then
\begin{equation}\label{Okt18aa}
\Psi_Y^2(X,\xi(X))\geq\frac{2a}{1+M^2},\;\;\mbox{where}\;\; a:=\sup_X (R-\xi(X)).
\end{equation}

(ii) If $R_0=\infty$ then
\begin{equation}\label{Okt18ab}
\Psi_Y(X,0)\geq \delta_1:=\sqrt{s^2-s_0^2},
\end{equation}
where $s\in (s_0,s_c)$ is the root of the equation ${\mathcal R}(s)=R$.

(iii) Let $R_0<\infty$, $\Omega(1)>0$ and $\Psi(X,\xi(X))\geq\delta>0$. Then
\begin{equation}\label{M4a}
\Psi_Y(X,0)>\theta_0\;\;\mbox{for $X\in\Bbb R$}.
\end{equation}
\end{lemma}
\begin{proof} The items (i) and (ii) are proved in Sect. 3 in \cite{Koz1b}.

Let us turn to (iii). 
We will use the  proof of Theorem 1' (c')-(d'), given in Sect.2.2.3, \cite{KLN17a},  with some small changes. The first step consists of the verification of the inequality
\begin{equation}\label{M20a}
h(q,p)\geq H(p;\theta)\;\;\mbox{with some $\theta>\theta_0$},
\end{equation}
where
$$
H(p;\theta)=\int_0^p\frac{d\tau}{\sqrt{\theta^2-2\Omega(\tau)}}.
$$
The proof of this inequality requires only that $\Psi_Y\geq 0$ and $\Psi_Y(X,\xi(X))\geq\delta $.
The inequality (\ref{M20a}) and the relations $H(0;\theta)=h(q,0)=0$ implies $H_p(0;\theta)\leq h_p(q,0)$, which gives
$$
\Psi_Y(X,0)\geq \theta_0.
$$
This completes the proof Lemma  \ref{LM28a}(iii).
\end{proof}

\begin{lemma}\label{Lokt19} Let the property (i) and one of the properties (ii) or (iii) of the previous lemma be valid. Then there exists  a constant $c_0>0$ depending on $M$, $\delta$, $R$ such that
\begin{equation}\label{Okt18a}
\Psi_Y(X,Y)\geq c_0\;\;\mbox{for all $(X,Y)\in D_\xi $}.
\end{equation}
\end{lemma}
\begin{proof} By Proposition 3.1\cite{KL1} there exists constants $\alpha\in (0,1)$ and $C>0$ depending on $M$, $\delta$ and $R$  such that
\begin{equation}\label{Okt18b}
||\Psi||_{C^{2,\alpha}(D_\xi)}\leq C.
\end{equation}
Now the estimate (\ref{Okt18a}) follows from (\ref{Okt18aa}), (\ref{Okt18ab}) and (\ref{Okt18b}) by using the Harnack principle.
\end{proof}

\begin{remark}\label{R21a} Using this lemma we can complete the proof Proposition \ref{Pr26}.
Indeed, this proposition is proved in \cite{Koz1b} except of the case $R_0<\infty$ and $\Omega(1)>1$.
Reference to Lemma \ref{Lokt19} gives the proof of the remaining part of Proposition \ref{Pr26} consisting of the case $R_0<\infty$ and $\Omega(1)>0$.

\end{remark}

The following property  is a consequence of Theorem 1, \cite{KLN17}.

\begin{proposition}\label{PM20a} Let $(\Psi,\xi,R)$ be a solitary wave solution to (\ref{K2a}) and let
\begin{equation}\label{F12a}
\sup|\xi'(X)|\leq M_1 \;\;\mbox{and}\;\;||\Psi(X,Y)||_{C^{2,\alpha}(S_\xi)}\leq M_2.
\end{equation}
Then there exists a positive $\epsilon$ depending on $M_1$ and $M_2$ such that if $R\in (R_c,R_c+\varepsilon)$ then $\Psi=\Psi(t)$, $\xi=\xi(t)$ and $R=R(t)$ for a certain small $t$.
\end{proposition}

\begin{proposition}\label{PrF12a} Let $\widehat{\mathcal U}_\delta$ and $\widehat{\mathcal U}$ be introduced by (\ref{M7a}) and (\ref{M7az}). The set $\overline{\widehat{{\mathcal U}}_\delta\bigcap\widehat{{\mathcal S}}}$ is compact in $C^{2,\alpha}_0(Q)\times [R_c,\infty)$.

\end{proposition}
\begin{proof} The result follows from Proposition \ref{PrF12az}.

\end{proof}

\subsection{Proofs of Theorems \ref{TrF18b} and \ref{ThF18}}

Theorem \ref{TrF18b} is a reformulation of Theorem \ref{ThF18} after application of partial hodograph transform. Theorem \ref{TrF18b} is more convenient for application of the abstract bifurcation analysis and Theorem \ref{ThF18} is more convenient for verification properties of global branches of bifurcating solutions (see properties a)-d) there).
Therefore we will prove both theorems simultaneously.

One can verify that the operator (\ref{Ja11aa}) satisfies the assumptions of Sect. \ref{SD1} and \ref{SectJ3} concerning Fredholm properties. The assumptions (\ref{3}) and (\ref{4}) follow from (\ref{Okt30a}). Furthermore, according to Theorem \ref{ThJ3} there exist a branch
$$
(w,\lambda)=(\kappa(s),\Lambda(s)):(0,\epsilon)\rightarrow {\mathcal U},
$$
whose asymptotics is described in that theorem. We denote this branch by ${\mathcal B}^+$. To obtain a global branch we intend to use Theorem \ref{ThJ3a}. We put
\[
{\mathcal K}_j=\widehat{\mathcal U}_{1/j},\;j=1,2,\ldots,\;\mbox{and}\;\;\widehat{\mathcal U}=\bigcup_{j}{\mathcal K}_j
\]
The compactness of ${\mathcal K}_j\bigcap {\mathcal S}$, required in that theorem,  follows from Proposition \ref{PrF12a}. It remains to verify the property (C) in the beginning of Sect. \ref{SectJ3} in order to apply Theorem \ref{ThJ3a}.
According to Proposition \ref{PrDa} we consider three cases  for the behavior of the branch ${\mathcal B}^+$.
We put
\begin{equation}\label{M25aa}
\widehat{h}(q,p;s)=h(q,p;t(s))+w(q,p;t(s)).
\end{equation}

\smallskip
1. Assume that the curve ${\mathcal B}^+$ satisfies the option (i) of Proposition \ref{PrDa}. Then the non-degeneracy condition (C) for the branch ${\mathcal B}^+$ is satisfied and
 we can apply Theorem \ref{ThJ3a} and get the existence of a curve of solitary solutions ${\mathcal B}$, continuously parameterised by $s\in [0,\infty)$, which extends ${\mathcal B}^+$. We denote by
   ${\mathcal J}^+$ and ${\mathcal J}$ the curves corresponding to ${\mathcal B}^+$ and ${\mathcal B}$ respectively, connected by (\ref{M25aa}). Both of them belongs to $\widehat{\mathcal U}\bigcap \widehat{\mathcal S}$.
According to Theorem \ref{ThJ3a} the assertions (a)-(d) imply all properties in Theorem \ref{TrF18b} except of (a)--(d) in Theorem \ref{T1}. Let us turn to proving these  properties.
Assume that non of the properties (a)-(d) in Theorem \ref{ThF18} holds. Then
$$
|\xi'(X)|\leq M,\;\;\Psi_Y\geq\delta>0,\;\;t(s)\leq T\;\;\mbox{along the curve ${\mathcal J}$}
$$
and the curve is not periodic. By Lemma \ref{LM20a} the inequality (\ref{M16a}) is valid and hence  Lemma \ref{LM20b} implies that
$$
||\Psi||_{C^{3,\alpha_1}(D_\xi)}\leq C_1,\;\;\||\xi||_{C^{3,\alpha_1}(S_\xi)}\leq C_1,
$$
where $C_1$ depends on the constants $M$, $\delta$ and $T$. By Proposition \ref{PM20a} there exist $\epsilon>0$ such that $R>R_c+\epsilon$. Otherwise the curve must coincide with the curve (\ref{J3b}) on a certain interval and hence everywhere, which contradict to the construction of the bifurcating curve. All these inequalities imply that the curve must belong to $\widehat{\mathcal U}_\delta$ for certain $\delta>0$. This implies that the closure of the curve is compact and since it is not periodic  this contradicts to the property (ii) in Theorem \ref{ThJ3a}.

\bigskip
2. Assume that the curve ${\mathcal B}^+$ satisfies the case (ii) in Proposition \ref{PrDa} and hence the property (i) in Theorem \ref{TrF18b}. This means that $\lambda\equiv 0$ or $t\equiv t_*$ on ${\mathcal B}^+$.

Introduce the following subspaces
$$
\widehat{\bf U}=\{(w,t)\in {\mathcal U}\,:\,P_0w=0\}\;\;\mbox{and}\;\;\widehat{{\mathcal Z}}=\{w\in {\bf Z}\times {\bf Y}\,:\,\widehat{P}_0w=0=0\},
$$
where the projector operators are given by (\ref{M24a}) and used in the Lyapunov-Schmidt method, see Sect. \ref{SD1}.
Consider the bifurcation problem
\begin{equation}\label{Dec29a}
{\mathcal F}_0(w,s):=(I-\widehat{P}_0)((A(0)w,{\mathcal N}(0)w)+\widehat{\mathcal F}_0(s\widehat{v}_0+w,0))=0, 
\end{equation}
which is the equation (\ref{F1ba}) with $\lambda=0$. Here ${\mathcal F}_0$ maps ${\bf U}$ into $\widehat{\mathcal Z}$.
 Then $w=\psi(s\tilde{v}_0,0)$ satisfies this equation and the equation (\ref{F1b}) which becomes
\begin{equation}\label{F1bx}
 \widehat{P}_0\widehat{{\mathcal F}}(s\widehat{v}+w,0)=0.
\end{equation}
We consider (\ref{Dec29a}) as a bifurcation equation.
 The corresponding Frechet derivative with respect to $w$ is invertible. So we can apply Theorem \ref{ThJ3a} to the bifurcation problem (\ref{Dec29a}), where
 $$
 {\mathcal B}^+=\{(w,s)\,:\, w=\psi(s\tilde{v}_0,0), \;s\in (0,\epsilon)\}.
 $$
The extension branch $(w,s)$ we denote by ${\mathcal B}$. The relation (\ref{F1bx}) is still valid for the whole curve ${\mathcal B}$ since this curve is analytic and it is valid on  ${\mathcal B}^+$. Thus the extension curve satisfies both equations (\ref{Dec29a}) and (\ref{F1bx}). The verification of properties (a)-(d) in Theorem \ref{TrF18b} can  be done in the same way as in the previous case 1.

\bigskip
3. Assume finally that the curve satisfies the case (iii) in Proposition \ref{PrDa}, i.e.
 $\lambda$ is not identically constant and $\mu_1(s)=0$ for small $s\neq 0$ and the bifurcating solution ${\mathcal B}^+$ is given by (\ref{Dec28a}) and (\ref{Dec28aa}).
Consider the curve $(x(s), \Lambda(s))$ from Theorem \ref{ThJ3}. This curve is analytic in an interval $(0,\epsilon)$ for a certain positive $\epsilon$. Let ${\mathcal J}^+$ be the curve obtained from ${\mathcal B}^+$ by (\ref{M25aa}).
Let us extend this curve to a global curve ${\mathcal J}\subset \widehat{\mathcal S}$ such that near every point in ${\mathcal J}$ it admits an analytic injective parameterisation. Due to local analyticity of the curve, $\mu_1=0$ on this curve and $\lambda$ can be zero only in a set consisting of isolated points of the curve. One of the end points of ${\mathcal J}$ is $(h(t_*),t_*)$.
Consider the set $\overline{\mathcal J}$. Assume that $\overline{\mathcal J}\in {\mathcal U}_\delta$. Then this set is compact by Proposition \ref{PrF12a}.  By continuity the eigenvalue, $\mu_1$ is equal to $0$ in all points in $\overline{\mathcal J}$.  Let us take a point $(\widehat{h}_1,\lambda_1)$ on the curve $\overline{\mathcal J}$.
Let us introduce the eigenfunction $\widehat{v}$ of the Frechet derivative at this point.
 We can assume that $L^2$ norm of $\widehat{v}$ is $1$.  Consider the equation
\begin{equation}\label{M22a}
{\mathcal F}(\widehat{h}_1+w,\lambda_1+\tau)=0
\end{equation}
in a neighborhood of this point.
By introducing the projectors $P$ and $\widehat{P}$ by (\ref{M24a}), we can represent $w=s\widehat{v}+\tilde{v}$ where $Pv=0$ and $s$ is a small number. Then the system (\ref{M22a}) can be split as follows
$$
\widehat{P}{\mathcal F}(\widehat{h}_1+s\widehat{v}+v,\lambda_1+\tau)=0
$$
and
$$
(I-\widehat{P}){\mathcal F}(\widehat{h}_1+s\widehat{v}+v,\lambda_1+\tau)=0.
$$
From the last system we can find  $v=v(s,\tau)$  by using the implicit function theorem, and write the first equation as
$$
{\bf F}(s,\tau):=P{\mathcal F}(\widehat{h}+s\widehat{v}+v(s,\tau),\lambda_1+\tau)=0.
$$
The function ${\bf F}(s,\tau)$ is analytic with respect  $s$  and  $\tau$ in a neighborhood of $(s,\tau)=(0,0)$. Let us show that it is not identical zero. Indeed, if it
is identical zero then the bifurcation equation (\ref{M22a}) is satisfied in a neighborhood of  $(\widehat{h}_1,\lambda_1)$. Since there are point of ${\mathcal J}$ inside this neighborhood we obtain a contradiction with the constraction of the curve ${\mathcal J}$. Therefore the function ${\bf F}(s,\tau)$ is not identically zero.
Then by Theorem \ref{ThJ3} there exist  local branches of local curves described there which admit a local analytic injective parametrization. Therefore the curve ${\mathcal J}$ can be extended if $(\widehat{h}_1,\lambda_1)$ does not belong to ${\mathcal J}$ or it is a loop. Thus if ${\mathcal J}$ is the maximal extension curve it must be a loop or it does not belong to any $\widehat{\mathcal U}_\delta$ with $\delta>0$. So it remain to prove the properties (a)-(d) in Theorem \ref{TrF18b}. This can be done in the same way as in the case 1.

\subsection{Proof of Theorem \ref{Tsep27}}\label{SSep28}

If $t_*$ is a turning point then the equation $R(t)=R$ has two roots $t_1(R)$ and $t_2(R)$ for $R\in (R-\epsilon,R)$ or $R\in (R,R+\epsilon)$ depending on the increasing-decreasing or decreasing-increasing behavior of the function $R(t)$. Then the vector functions
$$
(\Psi(t_k(R)),\xi(t_k(R)),R),\;\;k=1,2,
$$
give two different solitary solutions to the problem (\ref{K2a}).

If $t_*$ is not a turning point, i.e. the function $R(t)$ is strongly monotone in a neighborhood of $t_*$ then all bifurcations curves in Theorems \ref{TrF18b} and \ref{ThF18} give secondary bifurcations to the problem (\ref{K2a}) for the branch of solitary waves (\ref{J3b}).

Let $t(s)=t_*$ and hence $R(t(s))=R(t_*)$. Then the curve
$$
(\widehat{\Psi}_s,\widehat{\xi}_s,R(t_*)
$$
consists of solitary solutions to the problem (\ref{K2az}), which are different from the solution $(\Psi(t_*),\xi(t_*),R(t_*))$.

Now let $R(t)$ is a strongly monotone function for small $|t-t_*|$ and $t(s)$ is not identically $t_*$ for small $|s|$. Since the function $t(s)$ is analytic possibly after reparametrization, in a neighborhood of $s=0$ it is strongly monotone in  intervals $[0,\epsilon)$ and $(-\epsilon,0]$ for a small positive $\epsilon$. Four  cases are possible (i) $t(s)$ strongly increasing, (ii) it is strongly decreasing, (iii) it is increasing on $(-\epsilon,0]$ and then decreasing on $[0,\epsilon)$ and (iv) it is decreasing on $(-\epsilon,0]$ and then increasing on $[0,\epsilon)$.

Consider the case (i). Assume that it is strongly increasing. Then the equation $t(s)=t$ is uniquely solvable in a neighborhood of the point $s=0$, $t=t_*$ and denote by $s(t)$ its solution. The vector function
$$
(\widehat{\Psi}_{s(t)},\widehat{\xi}_{s(t)}, R(t))
$$
gives a solution to (\ref{K2az}) different from (\ref{J3b}), (\ref{Au30a}) in a neighborhood of $(\Psi(t_*),\xi(t_*),R(t_*))$. All other cases are considered similarly.
This completes the proof.

\bigskip
\noindent
{\bf Acknowledgements.} The study was carried out with the financial support of the Ministry of Science and Higher Education of the Russian Federation in the framework of a scientific project under agreement No.~075-15-2025-013.

The article has been accepted for publication in the Saint Petersburg Mathematical Journal.
The author expresses his appreciation to a anonymous reviewer for useful and important comments.




\section{References}

{

\end{document}